\documentclass[letterpaper,12pt]{article}

\usepackage[utf8]{inputenc}
\usepackage[english]{babel}

\usepackage{amsmath}
\allowdisplaybreaks
\usepackage{float}
\usepackage{latexsym}
\usepackage{amssymb}
\usepackage{amsthm}
\usepackage{mathrsfs}
\usepackage{amstext}
\usepackage{graphicx}
\usepackage{amsfonts}
\usepackage{wrapfig}
\usepackage{sidecap}
\usepackage{latexsym}
\usepackage{pdfpages}
\usepackage{epic}
\usepackage{fullpage}
\usepackage{color}
\usepackage{curves}
\usepackage{subfigure}
\usepackage{setspace}
\usepackage{amsmath}
\numberwithin{equation}{section}
\usepackage{anysize}
\usepackage{rotating}
\usepackage{enumerate}

\usepackage[ruled,vlined,linesnumbered,noresetcount]{algorithm2e}
\usepackage{hyperref}

\usepackage{bbm}
\usepackage{graphicx}
\usepackage{subfigure}

\newtheorem{theorem}{Theorem}[section]

\newtheorem{corollary}[theorem]{Corollary}
\newtheorem{proposition}[theorem]{Proposition}
\newtheorem{lemma}[theorem]{Lemma}
\newtheorem{remark}[theorem]{Remark}

\newcommand{\pt}{\partial_t}

\newcommand{\pnu}{\partial_\nu}

\renewcommand{\Re}{\operatorname{Re}}
\renewcommand{\Im}{\operatorname{Im}}

\title{Identification of source terms in the Schr\"odinger equation with dynamic boundary conditions from final data}

\date{}

\author{
	Salah-Eddine Chorfi\thanks{Faculty of Sciences Semlalia, Cadi Ayyad University, LMDP, IRD, UMMISCO, B.P. 2390, Marrakesh, Morocco, e-mail: s.chorfi@uca.ac.ma}
        \and
        Alemdar Hasanov\thanks{Department of Mathematics,
    Kocaeli University, 41001, Kocaeli, Turkey, e-mail: alemdar.hasanoglu@gmail.com}
	\and
	Roberto Morales\thanks{IMUS, Universidad de Sevilla, Campus Reina Mercedes, 41012 Sevilla, Spain, e-mail: rmorales1@us.es}
	}

\begin{document}
\maketitle

\begin{abstract}
    In this paper, we study an inverse problem of identifying two spatial-temporal source terms in the Schr\"odinger equation with dynamic boundary conditions from the final time overdetermination. We adopt a weak solution approach to solve the inverse source problem. By analyzing the associated Tikhonov functional, we prove a gradient formula of the functional in terms of the solution to a suitable adjoint system, allowing us to obtain the Lipschitz continuity of the gradient. Next, the existence and uniqueness of a quasi-solution are also investigated. Finally, our theoretical results are validated by numerical experiments in one dimension using the Landweber iteration method.
\end{abstract}

\smallskip


\tableofcontents

\section{Introduction}
The Schr\"odinger equation is a fundamental model in quantum mechanics which governs how quantum systems evolve over time. The equation describes the dynamical behavior of particles at the quantum level, such as electrons or atoms, by determining their wave functions; see e.g. \cite{BS12}. Linear and nonlinear Schrödinger equations have also been intensely studied due to their applications to plasma physics and laser optics; see e.g. \cite{Ba93} and \cite{GM92}. In the recent literature, an increasing interest has been devoted to inverse problems associated with the Schrödinger equation. However, the literature on source identification problems from the final time-measured data is scarce.

Let $\Omega\subset \mathbb{R}^{N} $ be a bounded domain $(N\geq 1)$ with boundary $\Gamma:=\partial \Omega$ of class $C^2$. For $T>0$, we adopt the notation $Q:=\Omega\times (0,T)$ and $\Sigma:=\Gamma\times (0,T)$. In the sequel, function spaces refer to spaces of complex-valued functions. Consider the following initial-boundary value problem for the Schr\"odinger equation with dynamic boundary conditions
\begin{align}
    \label{intro:eq:01}
    \begin{cases}
        i\pt y +d\Delta y-\vec{p}\cdot \nabla y+a(x)y=f(x,t)&\text{ in }Q,\\
        i\pt y_\Gamma -d\pnu y +\gamma \Delta_\Gamma y_\Gamma -\vec{p}_\Gamma \cdot \nabla_\Gamma y_\Gamma +a_\Gamma(x)y_\Gamma =f_\Gamma(x,t)&\text{ on }\Sigma,\\
        y\big|_{\Gamma}=y_\Gamma&\text{ on }\Sigma,\\
        (y(\cdot,0),y_\Gamma (\cdot,0))=(0,0)&\text{ in }\Omega\times\Gamma.
    \end{cases}
\end{align}
Here, $i=\sqrt{-1}$ is the imaginary unit, $d,\gamma>0$ are constant, $(a,a_\Gamma)\in L^\infty(\Omega)\times L^\infty(\Gamma)$ are potentials, $(\vec{p},\vec{p}_\Gamma)$ are drift vector fields, and $(f,f_\Gamma)\in L^2(Q)\times L^2(\Sigma)$ is the unknown source term. In addition, $y\big|_{\Gamma}$ denotes the trace of $y$ at the boundary $\Gamma$, $\pnu y=(\nabla y\cdot \nu)\big|_{\Gamma}$ is the normal derivative with respect to the outer unit normal field $\nu$ on $\Gamma$, and $\nabla_\Gamma y$ is the tangential gradient defined as $\nabla_\Gamma y=\nabla y-(\pnu y) \nu$ for a smooth function $y$. Due to the linearity of the problem, the pair of initial data $(y_0,y_{\Gamma,0})$ is assumed to be zero. Finally, $\Delta_\Gamma$ denotes the Laplace-Beltrami operator acting on $\Gamma$ (considered a compact Riemannian manifold).

In this paper, we study the \emph{Inverse Source Problem} {\bf (ISP)} of simultaneous identifying the pair of unknown spatial-temporal sources $f(x,t)$ and $f_{\Gamma}(x,t)$ in \eqref{intro:eq:01} from the following final time measured outputs
\begin{align}\label{(A1-1)}
\begin{cases}
u_T(x):=y(x,T), &x\in \Omega,\\
u_{T,\Gamma}(x):=y_{\Gamma}(x,T), &x\in \Gamma,
\end{cases}
\end{align}
where $y(x,T)\equiv y(x,t)\vert_{t=T}$ and $y_{\Gamma}(x,T)\equiv y_{\Gamma}(x,t)\vert_{t=T}$ are appropriately defined traces, $(y(x,t),y_\Gamma(x,t))$ is the weak solution of the initial-boundary value problem \eqref{intro:eq:01} corresponding to $(f(x,t), f_{\Gamma}(x,t))$, and $u_T(x)$, $u_{T,\Gamma}(x)$ represent the measured outputs containing a random noise.

Dynamic boundary condition models (such as \eqref{intro:eq:01}) have attracted increasing interest in recent years. They arise in many models of applied sciences and engineering such as heat transfer, fluid dynamics, the biology of cells, etc; see, for instance, \cite{Eg18,GM13,La32}. The reader may refer to \cite{Go'06} and \cite{Sa20} for physical interpretations of dynamical boundary conditions. In particular, the null controllability of parabolic equations has been considered in \cite{MMS'17}; the cubic Ginzburg-Landau equation in \cite{CMM23} and the Schr\"odinger equation in \cite{MM23}.

\subsection{Literature on inverse problems}
Here, we briefly review some relevant works related to inverse source problems of simultaneous determination of the source terms for evolution PDEs, in particular for the Schr\"odinger equation under consideration. The inverse source problems of simultaneous determination of source terms in linear parabolic and hyperbolic problems from the final overdetermination have been studied in \cite{Ha'07} and \cite{Ha091}. A weak solution approach combined with the quasi-solution and adjoint methods was proposed, and explicit gradient formulas for both problems were derived. This approach also provides a monotone iteration scheme to reconstruct unknown parameters, leading to rapid numerical reconstruction. One could also refer to \cite{CL0623}. It has been successfully applied to reconstruct forcing terms in wave equations \cite{CGMZ232}, Euler-Bernoulli beam equation \cite{Ha09,Ha19}, source terms and initial temperatures in heat equations with dynamic boundary conditions \cite{ACM'22,CGMZ23}. For a comprehensive study on this method, the reader may refer to the book \cite{HR'21}. It should be emphasized that this method has not been applied to the Schr\"odinger equation up to our knowledge.

In addition to the aforementioned, problems related to uniqueness and stability have also been extensively studied. The problem of determining a spatial potential from boundary measurement using global Carleman estimates, leading to Lipschitz stability, was investigated in \cite{BP08,BP02}. This technique has then been extended to a similar problem with a discontinuous principal coefficient in \cite{BM08}. In \cite{MOR08}, the authors showed stability results under less boundary or internal measurements using degenerate Carleman weights. Some one-dimensional inverse source and potential problems based on boundary measurements have been studied using the boundary control method \cite{AM12,AMR14}. In \cite{YFL19}, the authors determined the wave function from boundary measurements using an improved kernel method. Regularization and optimal error bounds are given with numerical examples. The authors of \cite{IY24} have proven the Lipschitz stability for real-valued source terms and potentials from data taken at the terminal time $T$ with a boundary measurement.

For inverse parabolic problems with dynamic boundary conditions, we refer to \cite{ACM'21,ACMO'21}, where the Lipschitz stability was proved for source terms and potentials, while the logarithmic stability was established for initial data. Finally, we refer to \cite{CCLL16} for the well-posedness and uniform stability of nonlinear Schr\"odinger equations.

\subsection{Notation and well-posedness of the system}
Throughout the paper, $\Re z$, $\Im z$, $\overline{z}$, and $|z|$ will denote the real part, the imaginary part, the conjugate, and the modulus of the complex number $z$, respectively.

We introduce the Banach space
\begin{align*}
    \mathbb{L}^p:=L^p(\Omega,dx)\times L^p(\Gamma,d\sigma), \qquad 1\le p\le \infty,
\end{align*}
where we denoted by $dx$ and $d\sigma$ the Lebesgue and surface measures on $\Omega$ and $\Gamma$, respectively. In the Hilbert space $\mathbb{L}^2$, we consider the inner product
$$\langle (y,y_\Gamma), (z,z_\Gamma)\rangle_{\mathbb{L}^2} =\Re \int_\Omega y \overline{z}\, dx + \Re \int_\Gamma y_\Gamma \overline{z_\Gamma}\, d\sigma.$$
For $k\in \mathbb{N}$ and $1\leq p\leq \infty$, we define the Banach space
\begin{align*}
    \mathbb{W}^{k,p}:=\left\{(y,y_\Gamma)\in W^{k,p}(\Omega)\times W^{k,p}(\Gamma)\,:\, y\big|_\Gamma =y_\Gamma\right\},
\end{align*}
equipped by the natural norm. In particular, for $k\in \mathbb{N}$, we write
\begin{align*}
    \mathbb{H}^k:=\mathbb{W}^{k,2}.
\end{align*}
We emphasize that, for all $k\in\mathbb{N}$, $\mathbb{H}^{k}$ is a Hilbert space with the inner product
\begin{align*}
    \langle (u,u_\Gamma), (v,v_\Gamma) \rangle_{\mathbb{H}^k} = \langle u,v \rangle_{H^k(\Omega)} + \langle u_\Gamma ,v_\Gamma \rangle_{H^k(\Gamma)}.
\end{align*}

To simplify the computation in the well-posedness of \eqref{intro:eq:01} and the Fréchet gradient formula, we henceforth assume that the coefficients $\vec{p}$ and $\vec{p}_\Gamma$ in \eqref{intro:eq:01} have the following form:
\begin{align}
    \label{assumptions:p:p:gamma}
    \vec{p}(x)=-i\vec{r}(x)\;\text{ in }\Omega\qquad\text{and}\qquad  \vec{p}_\Gamma(x)=-i\vec{r}_\Gamma(x)\;\text{ on }\Gamma,
\end{align}
where $(\vec{r},\vec{r}_\Gamma)\in [\mathbb{W}^{1,\infty}]^N$ is real-valued. Otherwise, a general class of drift coefficients can be recast to \eqref{assumptions:p:p:gamma}; see \cite{LTZ04} for more details.
We have the following result:
\begin{proposition}
    \label{proposition:wellposedness}
    Consider $d,\gamma>0$.
    \begin{itemize}
        \item[(a)] Suppose that $(a,a_\Gamma)\in \mathbb{L}^\infty$ and $(f,f_\Gamma)\in L^1(0,T;\mathbb{L}^2)$. Then, there exists a unique weak solution $(y,y_\Gamma)\in C^0([0,T];\mathbb{L}^2)$ of \eqref{intro:eq:01}. Moreover, there exists a constant $C>0$ independent of $(f,f_\Gamma)$ such that
        \begin{align*}
            \|(y,y_\Gamma)\|_{C^0([0,T];\mathbb{L}^2)}\leq C \|(f,f_\Gamma)\|_{L^1(0,T;\mathbb{L}^2)}.
        \end{align*}
        \item[(b)] Suppose that $(a,a_\Gamma)\in \mathbb{W}^{1,\infty}$ and $(f,f_\Gamma)\in L^2(0,T;\mathbb{H}^1)$. Then \eqref{intro:eq:01} has a unique solution $(y,y_\Gamma)\in C^0([0,T];\mathbb{H}^1)$ with extra regularity $\pnu y\in L^2(\Sigma)$. In addition, there exists a constant $C>0$ independent of $(f,f_\Gamma)$ such that 
        \begin{align*}
            \|(y,y_\Gamma)\|_{C^0([0,T];\mathbb{H}^1)}+ \|\partial_\nu y\|_{L^2(\Sigma)}\leq C \|(f,f_\Gamma)\|_{L^2(0,T;\mathbb{H}^1)}.
        \end{align*}
    \end{itemize}
\end{proposition}
The proof of Proposition \ref{proposition:wellposedness} relies on some elements of the semigroup theory. We sketch a proof in the Appendix \ref{appendix:A} for completeness.

\subsection{Input-output operator}
We introduce the \emph{input-output operators}
\begin{align}\label{(A2-1)}
\begin{cases}
(\Psi_1 F)(x):=y(x,T;F), &x\in \Omega,\\
(\Psi_2 F)(x) :=y_{\Gamma}(x,T;F), &x\in \Gamma,
\end{cases}
\end{align}
associated with the inverse source problem \eqref{intro:eq:01}-\eqref{(A1-1)}, where
\begin{equation}\label{(A2-2)}
F(x,t):= (f(x,t), f_{\Gamma}(x,t)) \in \mathcal{F},
\end{equation}
and $\mathcal{F} \subset L^2(0,T;\mathbb{L}^2)$ is an appropriately defined \emph{set of admissible spatial-temporal sources} that is bounded, closed and convex.

In view of these operators, the inverse source problem can be reformulated as the  following system of linear operator equations:
\begin{align}\label{(A2-3)}
\begin{cases}
(\Psi_1 F)(x)=u_T(x), &x\in \Omega,\\
(\Psi_2 F)(x)=u_{T,\Gamma}(x), &x\in \Gamma, \quad F\in \mathcal{F},
\end{cases}
\end{align}
where $u_T \in L^2(\Omega) $ and $u_{T,\Gamma} \in L^2(\Gamma) $ are the measured outputs defined in \eqref{(A1-1)}. We will also write $U_T:=(u_T,u_{T,\Gamma})$.

However, due to measurement errors in the outputs $u_T$ and $u_{T,\Gamma}$, exact equality in  the equations \eqref{(A2-3)} can not be satisfied. Hence one needs to introduce the Tikhonov functional
\begin{equation}\label{(A2-4)}
J(F)=\frac{1}{2}\, \Vert \Psi_1 F -u_T\Vert^2_{L^2(\Omega)}+ \frac{1}{2}\, \Vert
\Psi_2 F -u_{T,\Gamma}\Vert^2_{L^2(\Gamma)},
\end{equation}
and reformulate the inverse source problem \eqref{intro:eq:01}-\eqref{(A1-1)} as the following minimization problem
\begin{equation}
J(F^*)=\inf_{F \in \mathcal{F}}\, J(F).
\end{equation}
This is called the quasi-solution method; see Ivanov \cite{Iv63}.

\begin{proposition}
    \label{prop:input-output}
    The input-output operator $\Psi:L^2(0,T;\mathbb{H}^1) \to \mathbb{L}^2$ is compact.
\end{proposition}
\begin{proof}
    Let $(F_k)_{k\in \mathbb{N}}$ be a bounded sequence in $L^2(0,T;\mathbb{H}^1)$. Then, by Proposition \ref{proposition:wellposedness}, the sequence of associated weak solutions $Y(\cdot,\cdot,F_k)_{k\in\mathbb{N}}$ is bounded in $C^0([0,T];\mathbb{H}^1)$. In particular, the sequence $Y_k:=Y(\cdot,T,F_k)$ is bounded in $\mathbb{H}^1$. Thus, by the compact embedding $\mathbb{H}^1 \hookrightarrow \mathbb{L}^2$, there exists a subsequence of $(Y_k)_{k\in\mathbb{N}}$ which converges in $\mathbb{L}^2$. This implies that the input-output operator $\Psi$ is compact.
\end{proof}


From Proposition \ref{prop:input-output}, it is clear that the inverse problem {\bf (ISP)} is ill-posed in the sense of Hadamard. Due to the ill-posedness of \eqref{(A2-3)}, one usually considers a regularized version of \eqref{(A2-4)} (e.g. a Tikhonov regularization). More precisely, for a regularizing parameter $\epsilon>0$, we introduce the regularized functional
\begin{align}
J_\epsilon(F)=\frac{1}{2}\, \Vert \Psi_1 F -u_T\Vert^2_{L^2(\Omega)}+ \frac{1}{2}\, \Vert
\Psi_2 F -u_{T,\Gamma}\Vert^2_{L^2(\Gamma)}+\frac{\epsilon}{2} \|F\|^2_{L^2(0,T;\mathbb{L}^2)},\quad F\in\mathcal{F}.
\end{align}
In the sequel, we omit the Tikhonov regularization, leveraging the regularization property of the adjoint Landweber iteration (see \cite{Nem86}). The reader may refer to \cite{EHN'00} for regularization of ill-posed problems.

\subsection*{Outline of the paper}
The rest of the article is organized as follows. In Section \ref{section:Frechet:differentiability}, we focus on the properties of the Tikhonov functional, that is, we characterize its Fr\'echet derivative via a suitable adjoint problem. In Section \ref{section:Existence:quasi}, we give sufficient conditions for the existence and uniqueness of quasi-solutions to {\bf (ISP)}. In Section \ref{section:numerical:experiments}, we illustrate our theoretical findings by some numerical experiments to reconstruct source terms in the 1D system. In Section \ref{section_further:comments}, we give additional comments concerning the theoretical and numerical results obtained in this paper. Finally, Appendix \ref{appendix:A} is devoted to the well-posedness of the system \eqref{intro:eq:01}.
\section{Fr\'echet differentiability and gradient formula}
\label{section:Frechet:differentiability}
Now, we derive a gradient formula for the Tikhonov functional $J$ which is the key result for the sequel.
\begin{proposition}
    \label{prop:gradient:J}
    The Tikhonov functional $J\colon \mathcal{F}\rightarrow \mathbb{R}$ is Fr\'echet differentiable and its gradient at each $F\in \mathcal{F}$ is given by
    \begin{align}
        \label{J:prime:phi}
        J'(F)=\Phi,
    \end{align}
    where $\Phi(\cdot,\cdot;F)=(\phi,\phi_\Gamma)$ is the unique weak solution of the following adjoint system
    \begin{align}\label{adjeq}
        \begin{cases}
            i\pt \phi +d\Delta \phi + i\vec{r}\cdot \nabla \phi +b \phi =0&\text{ in }Q,\\
            i\pt \phi_\Gamma -d\pnu \phi +\gamma \Delta_\Gamma \phi_\Gamma + i\vec{r}_\Gamma\cdot\nabla_\Gamma \phi_\Gamma  +b_\Gamma \phi_\Gamma =0&\text{ on }\Sigma,\\
            \phi\big|_{\Gamma}=\phi_\Gamma&\text{ on }\Sigma,\\
            (\phi(\cdot,T),\phi_\Gamma (\cdot,T))=i(Y(\cdot,T;F)-U_T),&\text{ in }\Omega\times \Gamma,
        \end{cases}
    \end{align}
    where the coefficients $b$ and $b_\Gamma$ are given by
    \begin{align}
        \label{def:b:b:gamma}
        b:=i\mathrm{div}(\vec{r})+\overline{a},\qquad b_\Gamma:= i\mathrm{div}(\vec{r}_\Gamma)-i\vec{r}\cdot \nu +\overline{a}_\Gamma.
    \end{align}
\end{proposition}

\begin{proof}
    Consider $F,\delta F\in \mathcal{F}$ such that $F+\delta F\in \mathcal{F}$. Let us compute the difference
    \begin{align*}
        \delta J(F):=J(F+\delta F)- J(F).
    \end{align*}
Then, straightforward computations show that
    \begin{align*}
        \delta J(F)&=\dfrac{1}{2}\|Y(\cdot,T;F+\delta F)- U_T\|_{\mathbb{L}^2}^2 -\dfrac{1}{2}\| Y(\cdot,T;F)-U_T\|_{\mathbb{L}^2}^2\\
        &= \dfrac{1}{2}\int_\Omega (|y(\cdot,T;F+\delta F) - u_T|^2 - |y(\cdot,T;F) - u_T|^2)\,dx\\
        &+\frac{1}{2}\int_{\Gamma} (|y_\Gamma(\cdot,T;F+\delta F)-u_{T,\Gamma}|^2 - |y_\Gamma (\cdot,T;F)-u_{T,\Gamma}|^2)\,d\sigma.
    \end{align*}
Now, using the complex identity
    \begin{align*}
        \frac{1}{2}(|x-z|^2-|y-z|^2)=\Re [(x-y)\overline{(y-z)}] +\dfrac{1}{2}|x-y|^2\qquad \forall x,y,z\in \mathbb{C},
    \end{align*}
we have
    \begin{align}
        \nonumber
        \delta J(F)=& \Re\int_\Omega (y(\cdot,T;F+\delta F)-y(\cdot,T;F))\overline{(y(\cdot,T;F)-u_T})\,dx \\
        \nonumber
        &+ \dfrac{1}{2}\int_\Omega |y(\cdot,T;F+\delta F)-y(\cdot,T;F)|^2\,dx\\
        \nonumber
        &+\Re \int_\Gamma (y_\Gamma (\cdot,T;F+\delta F)-y_\Gamma (\cdot,T;F))\overline{(y_\Gamma (\cdot,T;F)-u_{T,\Gamma})}\, d\sigma \\
        \nonumber
        &+\dfrac{1}{2}\int_{\Gamma} |y_\Gamma (\cdot,T;F+\delta F)-y_\Gamma (\cdot,T;F)|^2\,d\sigma ,\\
        \label{delta:J:F:01}
        \begin{split}
        =& -\Im\int_\Omega \delta y(\cdot,T;F) \overline{\phi (\cdot,T;F)}\,dx + \dfrac{1}{2}\int_\Omega |\delta y(\cdot,T;F)|^2\,dx\\
        &-\Im  \int_{\Gamma } \delta y_\Gamma (\cdot,T;F) \overline{\phi_\Gamma(\cdot,T;F)}\,d\sigma +\dfrac{1}{2}\int_{\Gamma} |\delta y_\Gamma (\cdot,T;F)|^2\,d\sigma,
        \end{split}
    \end{align}
    where $\delta Y=(\delta y,\delta y_\Gamma)$ is the solution of the following sensitivity system
    \begin{align}
        \label{problem:delta:y}
        \begin{cases}
            i\pt (\delta y)+d\Delta (\delta y) +i\vec{r}\cdot \nabla (\delta y) +a(x)(\delta y)=\delta f&\text{ in }Q,\\
            i\pt (\delta y_\Gamma) -d\pnu (\delta y)+\gamma \Delta_\Gamma (\delta y_\Gamma) +i\vec{r}_\Gamma \cdot \nabla_\Gamma (\delta y_\Gamma) +a_\Gamma(x)(\delta y_\Gamma) =\delta f_\Gamma&\text{ on }\Sigma,\\
            \delta y\big|_{\Gamma}=\delta y_\Gamma &\text{ on }\Sigma,\\
            (\delta y(\cdot,0),\delta y(\cdot,0))=(0,0)&\text{ in }\Omega\times \Gamma.
        \end{cases}
    \end{align}
Then by integration by parts, it is clear that \eqref{delta:J:F:01} can be written as
    \begin{align}
        \label{delta:J:F:02}
        \begin{split}
        \delta J(F)=&\Re \iint_{Q} \delta f\overline{\phi}\,dx\,dt + \Re \iint_{\Sigma} \delta f_\Gamma \overline{\phi_\Gamma}\,d\sigma\,dt +  \dfrac{1}{2}\int_{\Omega} |\delta y(\cdot,T;F)|^2\,dx\\
        &+\dfrac{1}{2}\int_\Gamma |\delta y_\Gamma (\cdot,T;F)|^2\,d\sigma .
        \end{split}
    \end{align}
Therefore, the equality \eqref{delta:J:F:02} implies the identity \eqref{J:prime:phi}.
\end{proof}

Next, we prove a crucial lemma when dealing with gradient iterative methods.
\begin{lemma}
    \label{Lemma:Frechet:lipschitz}
     The Fr\'echet gradient $J'$ is Lipschitz continuous, i.e., there exists a constant $L>0$ such that for all $F,\delta F\in \mathcal{F}$ such that $F+\delta F\in \mathcal{F}$,
     \begin{align}
        \label{eq:Lipchitz:gradient}
         \|J'(F+\delta F)-J'(F)\|_{L^2(0,T;\mathbb{L}^2)}\leq L \|\delta F\|_{L^2(0,T;\mathbb{L}^2)}.
     \end{align}
\end{lemma}

\begin{proof}
    Fix $F, \delta F\in \mathcal{F}$ such that $F+\delta F\in \mathcal{F}$. Then, $\delta \Phi (\cdot,T;F)=(\delta \phi,\delta\phi_\Gamma)$ is the solution of the adjoint system
    \begin{align}
        \begin{cases}
            i\pt (\delta \phi)+d\Delta (\delta \phi)+i\vec{r}\cdot \nabla (\delta \phi) +b(x)(\delta \phi)=0,&\text{ in }Q,\\
            i\pt (\delta \phi_\Gamma)-d\pnu (\delta \phi)+\gamma \Delta_\Gamma (\delta \phi_\Gamma)+i\vec{r}_\Gamma\cdot \nabla_\Gamma (\delta \phi_\Gamma) + b_\Gamma(x)(\delta \phi_\Gamma)=0&\text{ on }\Sigma,\\
            \delta \phi \big|_{\Gamma}=\delta \phi_\Gamma&\text{ on }\Sigma,\\
            (\delta \phi,\delta \phi_\Gamma)(\cdot,T)=i(\delta y(\cdot,T;F),\delta y(\cdot,T;F))&\text{ in }\Omega\times \Gamma,
        \end{cases}
    \end{align}
    where $\delta Y:=(\delta y(\cdot,\cdot;F), \delta y_\Gamma (\cdot,\cdot;F))$ is the solution of \eqref{problem:delta:y} associated to $\delta F$ and $(b,b_\Gamma)$ are the potentials defined in \eqref{def:b:b:gamma}. By Proposition \ref{prop:gradient:J}, we have for some generic constant $C>0$,
    \begin{align*}
        \|J'(F+\delta F)-J'(F)\|_{L^2(0,T;\mathbb{L}^2)}=\|\delta \Phi\|_{L^2(0,T;\mathbb{L}^2)}\leq C \|\delta Y(\cdot,T;F)\|_{\mathbb{L}^2}\leq C\|\delta F\|_{L^2(0,T;\mathbb{L}^2)},
    \end{align*}
    where we have used Proposition \ref{proposition:wellposedness} (part (a)) applied to $\delta Y=(\delta y,\delta y_\Gamma)$. This ends the proof of Lemma \ref{Lemma:Frechet:lipschitz}.
\end{proof}


\section{Existence and uniqueness of the solution to (ISP)}
\label{section:Existence:quasi}
Using the gradient formula \eqref{J:prime:phi} and the sensitivity system \eqref{problem:delta:y}, one can easily prove the monotonicity of the derivative $J':\mathcal{F}\to \mathbb{R}$. This implies the convexity of the functional $J$. Consequently, we have the following existence result.
\begin{corollary}
    The Tikhonov functional $J$ is continuous and convex on $\mathcal{F}$. Then, there exists a minimizer $F^*\in \mathcal{F}$ such that
    \begin{align*}
        J(F^*)=\min_{F\in \mathcal{F}} J(F).
    \end{align*}
\end{corollary}
Since the strict convexity of $J$ is characterized by the strict monotonicity of $J'$, straightforward computation yields the following sufficient condition for uniqueness.
\begin{lemma}
    If the positivity condition
    \begin{align*}
        \int_\Omega |\delta y(\cdot,T;F)|^2\,dx + \int_{\Gamma} |\delta y_\Gamma (\cdot,T;F)|^2\,d\sigma >0\qquad \forall F\in \mathcal{V},
    \end{align*}
    holds on a closed convex subset $\mathcal{V}\subset \mathcal{F}$, then the problem {\bf (ISP)} admits at most one solution in $\mathcal{V}$.
\end{lemma}

\section{Numerical experiments for the 1D case}
\label{section:numerical:experiments}
In this section, we examine the numerical reconstruction of a spatial forcing term. Let $\ell>0$ and consider reconstructing a complex-valued function $f(x)$ in the 1-D Schr\"odinger equation with dynamic boundary conditions
\begin{align}\label{1deq1to4}
        \begin{cases}
            i y_t + y_{xx} =f(x) g(x,t)&\qquad(x,t)\in (0,\ell)\times (0,T),\\
            i y_t(0,t)+ y_x(0,t)=0 &\qquad t\in (0,T), \\
            i y_t(\ell,t)- y_x(\ell,t)=0 &\qquad t\in (0,T),\\
            (y,y(0,\cdot),y(\ell,\cdot))\rvert_{t=0}=(0,0,0)&\qquad x\in (0,\ell),
        \end{cases}
    \end{align}
where the function $g \in L^2\left(0,T; L^\infty(0,\ell)\right)$ is known. This special form of forcing terms is usually required to ensure the uniqueness of the inverse problem (see, e.g., \cite{AM12}).
\begin{remark}
Here we consider the reconstruction of one internal source term $f(x)$, because boundary source terms $f_\Gamma(x)$ are constant in this 1D case. For simplicity, we omit the reconstruction of constant functions.
\end{remark}
Let $Y(x,t;f):=\left(y(x,t),y(0,t),y(\ell,t)\right)$ be the solution to \eqref{1deq1to4}. Recall that the input-output operator $\Psi \colon L^2(0,\ell) \longrightarrow L^2(0,\ell)\times \mathbb{C}^2$ is defined by
$$(\Psi f)(x):=Y(x,T;f)=\left(y(x,T),y(0,T),y(\ell,T)\right), \qquad x\in (0,\ell),$$
and the Tikhonov functional is given by
\begin{align*}
J(f)&=\frac{1}{2} \left\|Y(\cdot,T;f)-U_T \right\|_{L^2(0,\ell)\times \mathbb{C}^2}^2, \qquad f\in L^2(0,\ell), \\
& \hspace{-0.3cm}= \frac{1}{2} \left(\left\|y(\cdot,T)-u_T\right\|_{L^2(0,\ell)}^2 + \left|y(0,T)-u_T^0\right|^2 + \left|y(\ell,T)-u_T^\ell\right|^2\right), \notag
\end{align*}
where $U_T:=\left(u_T, u_T^0, u_T^\ell\right) \in L^2(0,\ell)\times \mathbb{C}^2$.
The adjoint system associated with \eqref{1deq1to4} reads as follows
\begin{align}\label{1daeq1to4}
        \begin{cases}
            i \phi_t + \phi_{xx} =0&\qquad(x,t)\in (0,\ell)\times (0,T),\\
            i\phi_t(0,t)+ \phi_x(0,t)=0 &\qquad t\in (0,T), \\
            i\phi_t(\ell,t)- \phi_x(\ell,t)=0 &\qquad t\in (0,T),\\
            \phi(x,T)=i(y(x,T)-u_T(x)) & \qquad x\in (0,\ell),\\(\phi(0,T),\phi(\ell,T))=i\left(y(0,T)-u_T^0,y(\ell,T)-u_T^\ell\right)&\qquad x\in (0,\ell).
        \end{cases}
    \end{align}
In a similar manner to the computations in Section \ref{section:Frechet:differentiability}, the gradient of $J$ can be expressed as
\begin{equation}\label{1dj'}
J'(f)(x)=\int_0^T \phi(x,t;f) g(x,t) \,d t, \quad f\in L^2(0,\ell),\;  x\in (0,\ell).
\end{equation}
This gradient formula enables the implementation of gradient-type iterative methods. We shall use the Landweber iteration (Algorithm \ref{alg1}). We refer to \cite[\S 3.4.3]{HR'21} for more details.
\begin{algorithm}[H]
\caption{Landweber iteration}\label{alg1}
 Set $k=0$ and choose an initial guess $f_0$\;
 Solve the forward problem \eqref{1deq1to4} to obtain $Y(x,t;f_k)$\;
 Knowing the computed $Y(x,T;f_k)$ and the measured $U_T$, solve the adjoint system \eqref{1daeq1to4} to obtain $\phi(x,t;f_k)$\;
 Compute the descent direction $p_k=J'(f_k)$ using \eqref{1dj'}\;
 Solve the forward problem \eqref{1deq1to4} with source $f=p_k$ to get the solution $\Psi p_k$\;
 Compute the step parameter $\displaystyle \alpha_k =\frac{\|p_k\|_{L^2(0,\ell)}^2}{\|\Psi p_k\|_{L^2(0,\ell)\times \mathbb{C}^2}^2}$\;
 Compute the next iteration $f_{k+1}=f_k- \alpha_k p_k$\;
 Stop the algorithm if the stopping criterion $J(f_{k+1}) <e_J$ holds for a given tolerance $e_J$. Otherwise, set $k:=k+1$ and run Step 2\;
\end{algorithm}
The algorithm will be implemented using the \texttt{Wolfram} language to solve the forward and backward PDEs using the method of lines. We discretize the spatial interval $[0,\ell]$ into a uniform grid $(x_j)_{j=0}^{N_x}$ with step size $h=\frac{\ell}{N_x}$ using the finite difference method. Noting $y_j(t):=y(x_j,t)$, the second derivative is approximated by
$$y_{xx}(x_j,t) \approx \frac{y_{j-1}(t)-2 y_j(t) + y_{j+1}(t)}{h^2}, \qquad j=\overline{1, N_x-1}.$$
The first derivatives on the boundary are approximated by
\begin{align*}
    y_x(0,t) &\approx \frac{y_1(t)-y_0(t)}{h},\\
    y_x(\ell,t) &\approx \frac{y_{N_x}(t)-y_{N_x-1}(t)}{h}.
\end{align*}
Then, we solve the resulting system of ordinary differential equations numerically.

The noisy terminal data will be generated from the exact output data as follows
$$U_T(x) =Y_T(x) + p \times\|Y_T\|_{L^2(0,\ell)\times \mathbb{C}^2} \times \text{(random number)},$$
where $p$ represents the percentage of noise level.

In subsequent numerical experiments, we choose the following values
$$T=1, \quad \ell=1, \quad N_x=25, \quad g=1, \quad f_0=0.$$

\noindent\textbf{Example 1.} (Real source term)\\
We take the real source term $f(x)=\sin(\pi x), \; x\in (0,1)$. Next, we plot the corresponding solution to \eqref{1deq1to4} (Figure \ref{fig1}).

For the stopping parameter $e_J=10^{-6}$, the algorithm stops at iterations $k\in \{2,2,2\}$, for $p\in \{1\%,3\%,5\%\}$, respectively.
\begin{figure}[H]
    \centering
    {{\includegraphics[width=7.3cm]{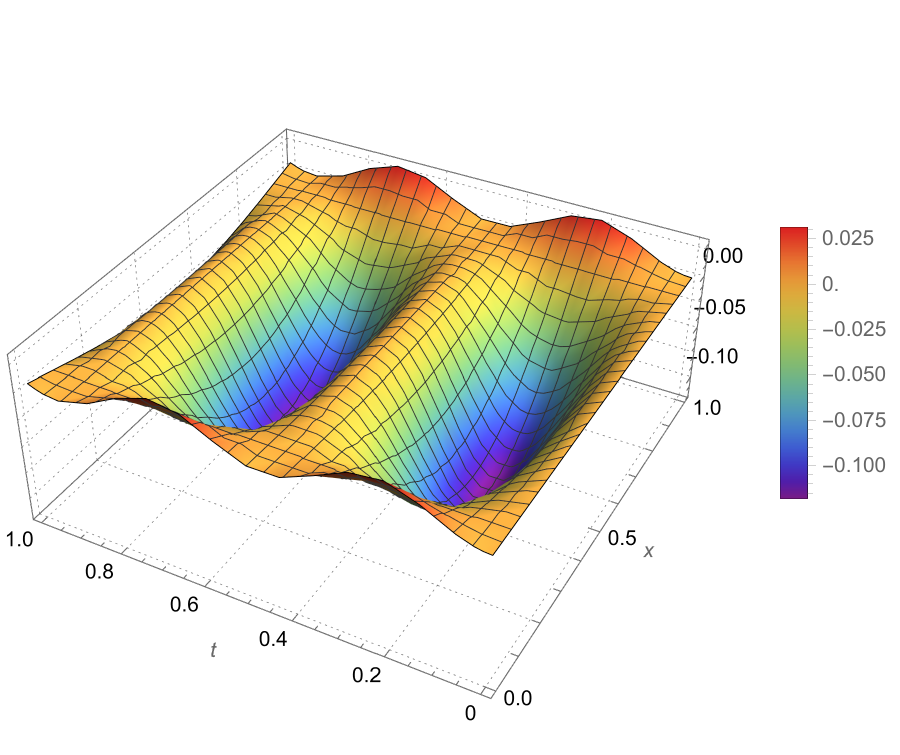} }}%
    \hspace{0.5cm}
    {{\includegraphics[width=7.3cm]{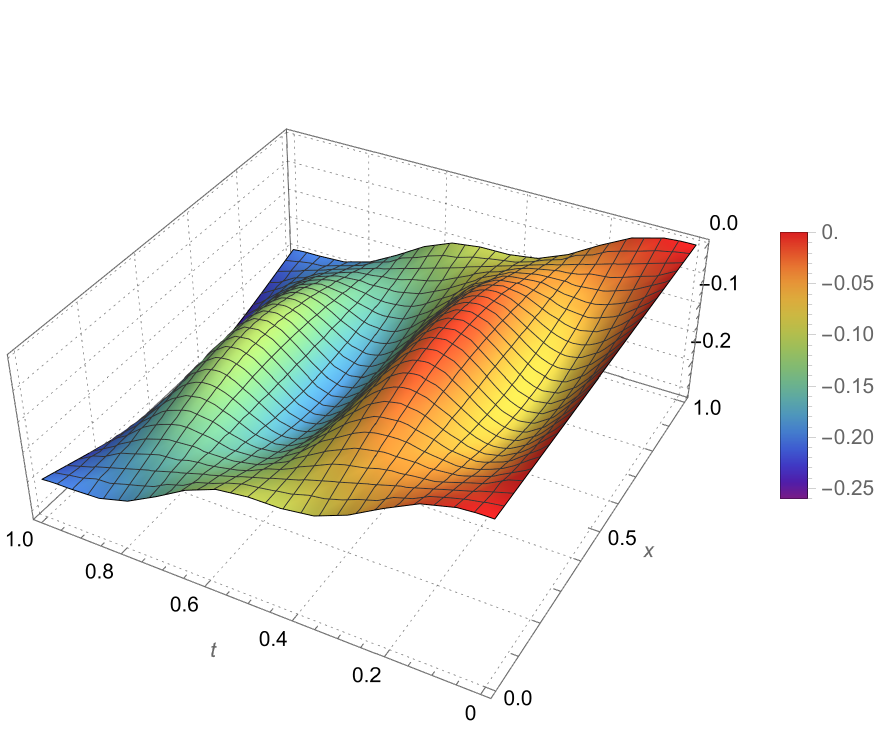} }}%
    \caption{Real part (left) and imaginary part (right) of the solution $Y(x,t;f)$ to \eqref{1deq1to4}.}%
    \label{fig1}%
\end{figure}

\begin{figure}[H]
    \centering
    {{\includegraphics[width=7.3cm]{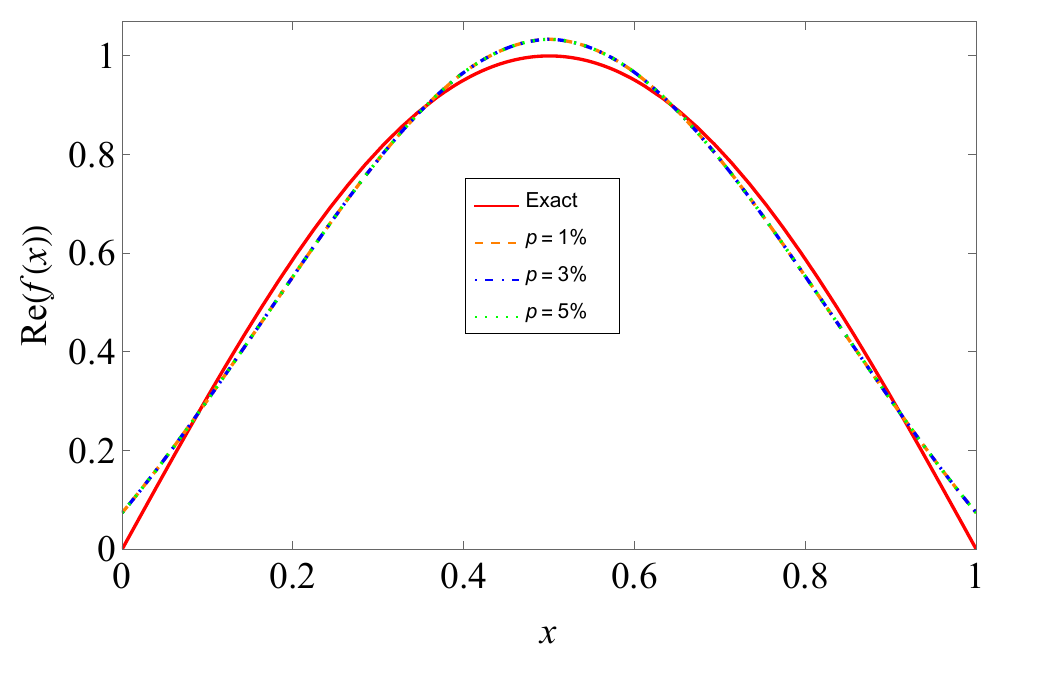} }}%
    \hspace{0.5cm}
    {{\includegraphics[width=7.3cm]{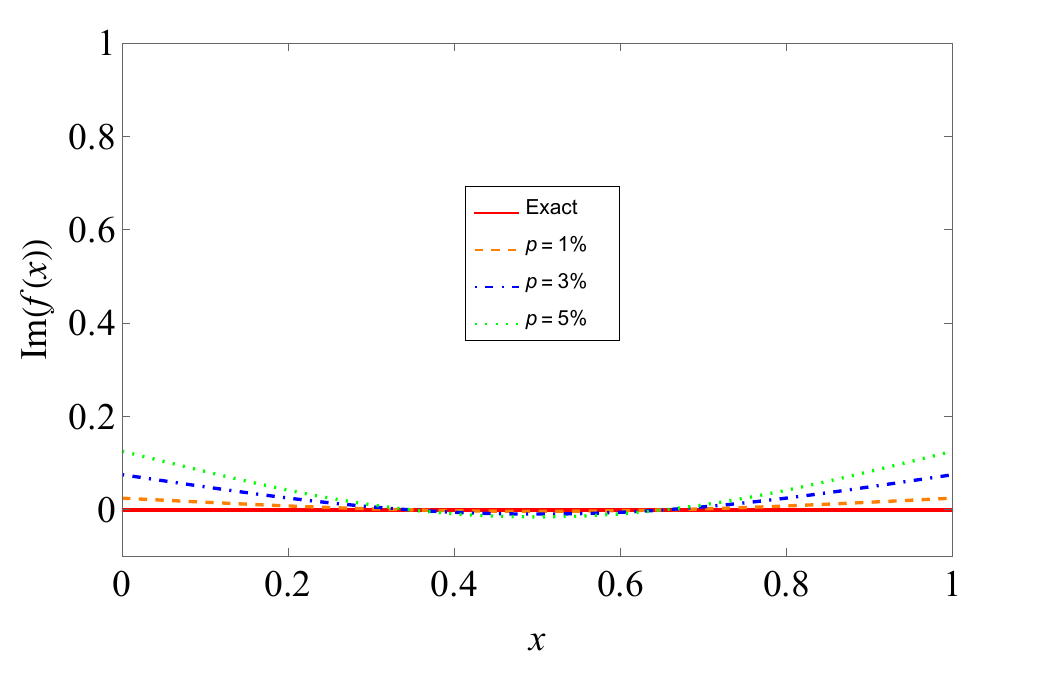} }}%
\caption{Exact and recovered $\Re f(x)$ (left) and $\Im f(x)$ (right) for $p\in \{1\%,3\%,5\%\}$.}\label{fig2}
\end{figure}
\bigskip

\noindent\textbf{Example 2.} (Imaginary source term)\\
We choose the imaginary source term $f(x)=ix (1-x), \; x\in (0,1)$. Next, we plot the corresponding solution to \eqref{1deq1to4} (Figure \ref{fig3}).

For the stopping parameter $e_J=10^{-8}$, the algorithm stops at iterations $k\in \{27,48,81\}$, for $p\in \{1\%,3\%,5\%\}$, respectively.
\begin{figure}[H]
    \centering
    {{\includegraphics[width=7.3cm]{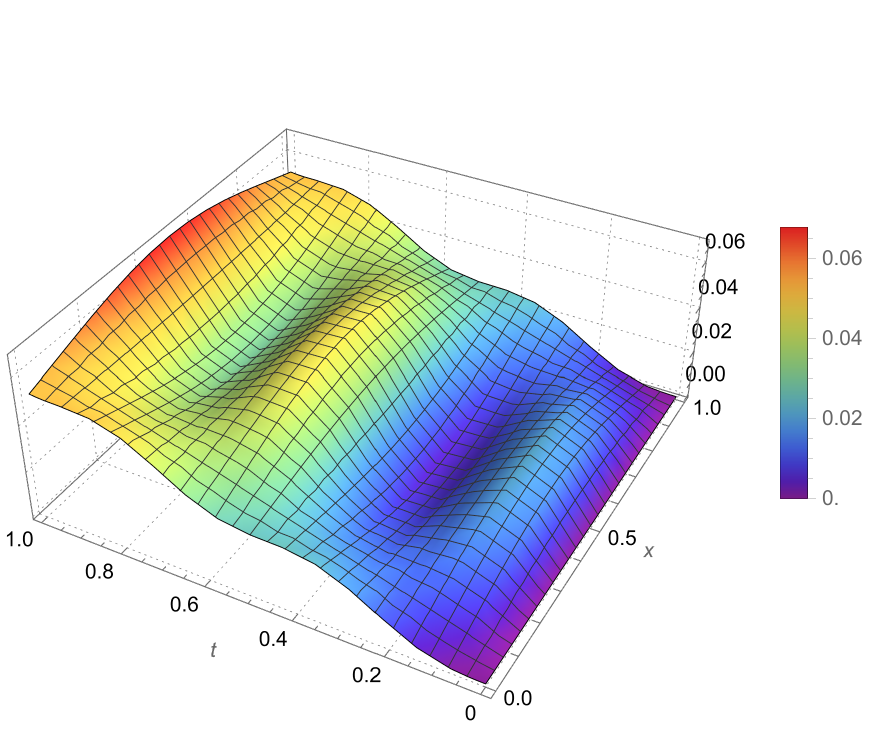} }}%
    \hspace{0.5cm}
    {{\includegraphics[width=7.3cm]{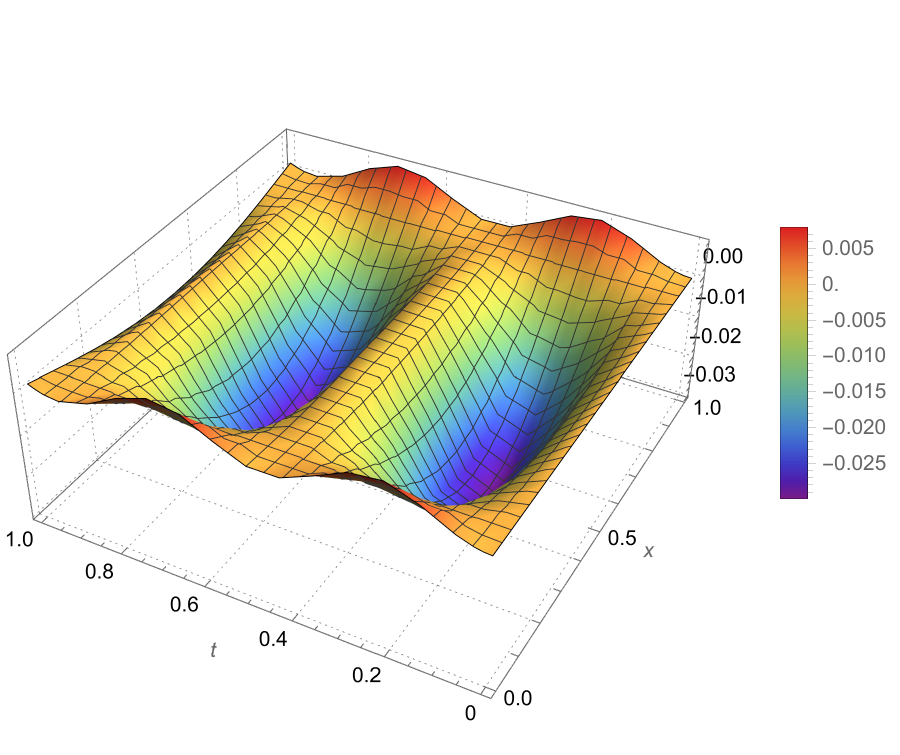} }}%
    \caption{Real part (left) and imaginary part (right) of the solution $Y(x,t;f)$ to \eqref{1deq1to4}.}%
    \label{fig3}%
\end{figure}

\begin{figure}[H]
    \centering
    {{\includegraphics[width=7.3cm]{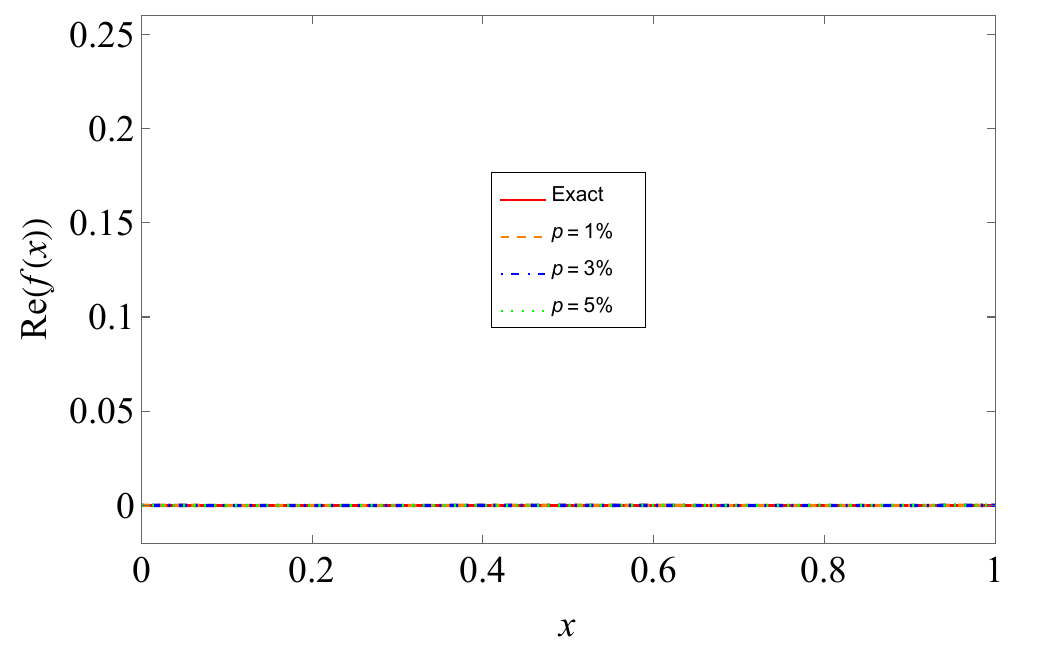} }}%
    \hspace{0.5cm}
    {{\includegraphics[width=7.3cm]{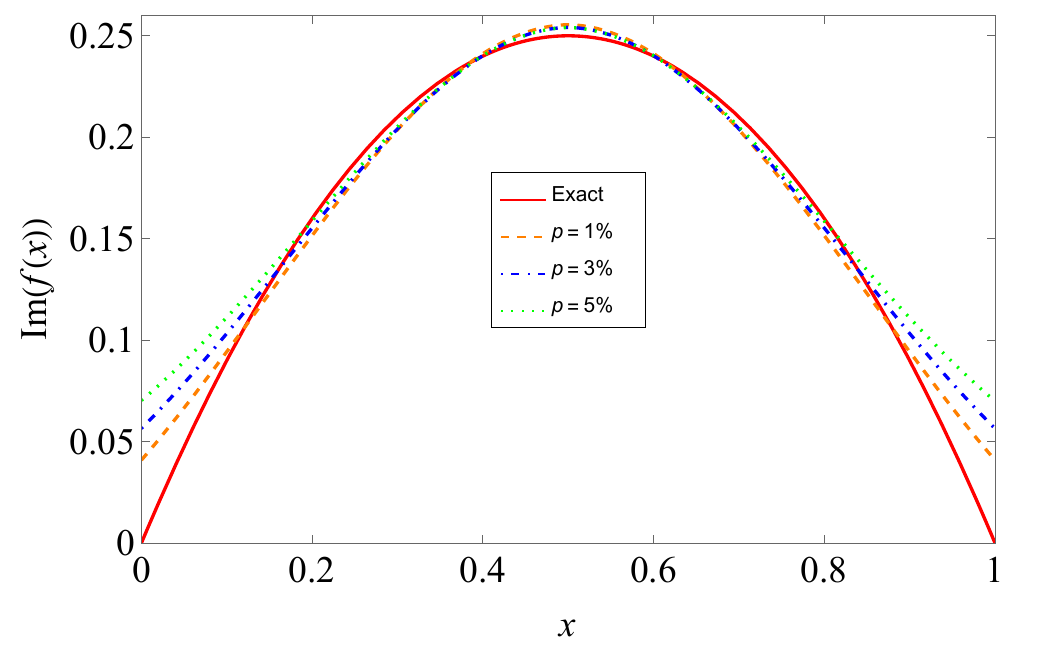} }}%
\caption{Exact and recovered $\Re f(x)$ (left) and $\Im f(x)$ (right) for $p\in \{1\%,3\%,5\%\}$.}\label{fig4}
\end{figure}
\bigskip

\noindent\textbf{Example 3.} (Complex source term)\\
We take the exact source term $f(x)=e^{i \pi x}, \; x\in (0,1)$. Next, we plot the corresponding solution to \eqref{1deq1to4} (Figure \ref{fig5}).

For the stopping parameter $e_J=10^{-8}$, the algorithm stops at iterations $k\in \{309,333,370\}$ for $p\in \{1\%,3\%,5\%\}$, respectively.
\begin{figure}[H]
    \centering
    {{\includegraphics[width=7.3cm]{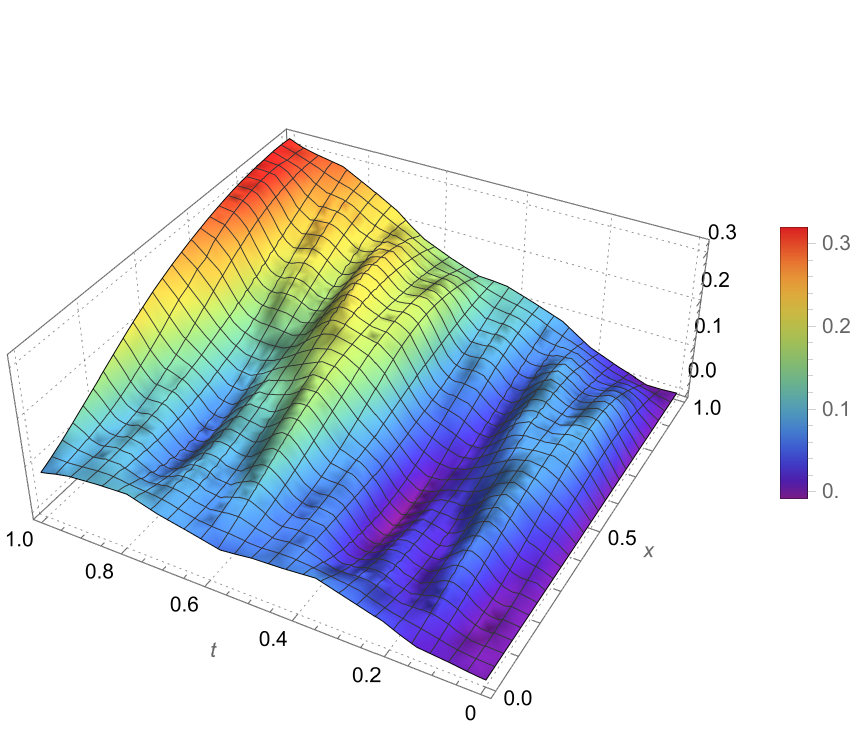} }}%
    \hspace{0.5cm}
    {{\includegraphics[width=7.3cm]{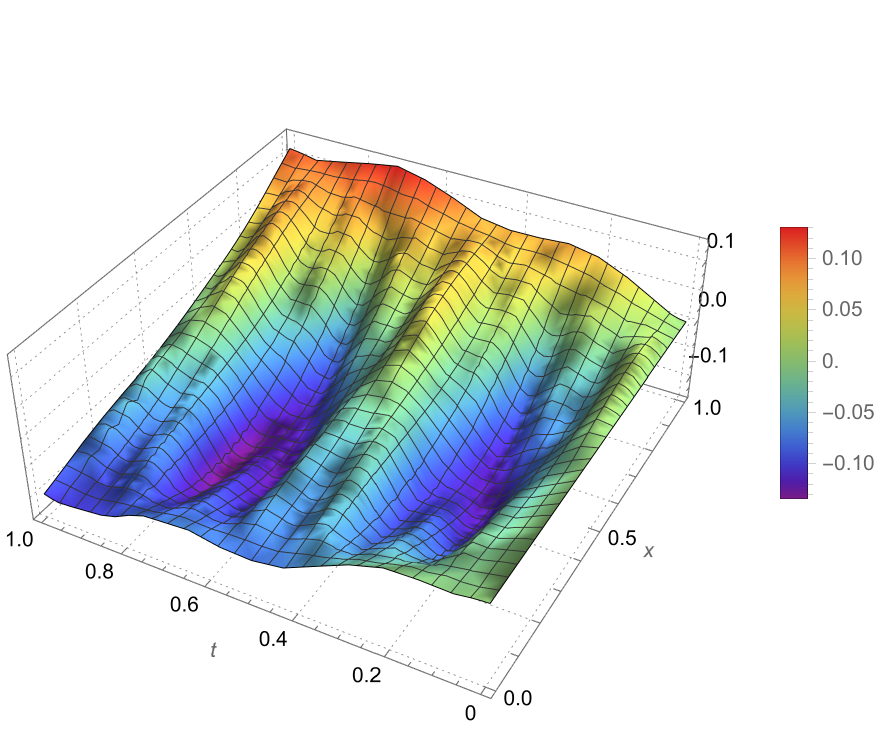} }}%
    \caption{Real part (left) and imaginary part (right) of the solution $Y(x,t;f)$ to \eqref{1deq1to4}.}%
    \label{fig5}%
\end{figure}

\begin{figure}[H]
    \centering
    {{\includegraphics[width=7.3cm]{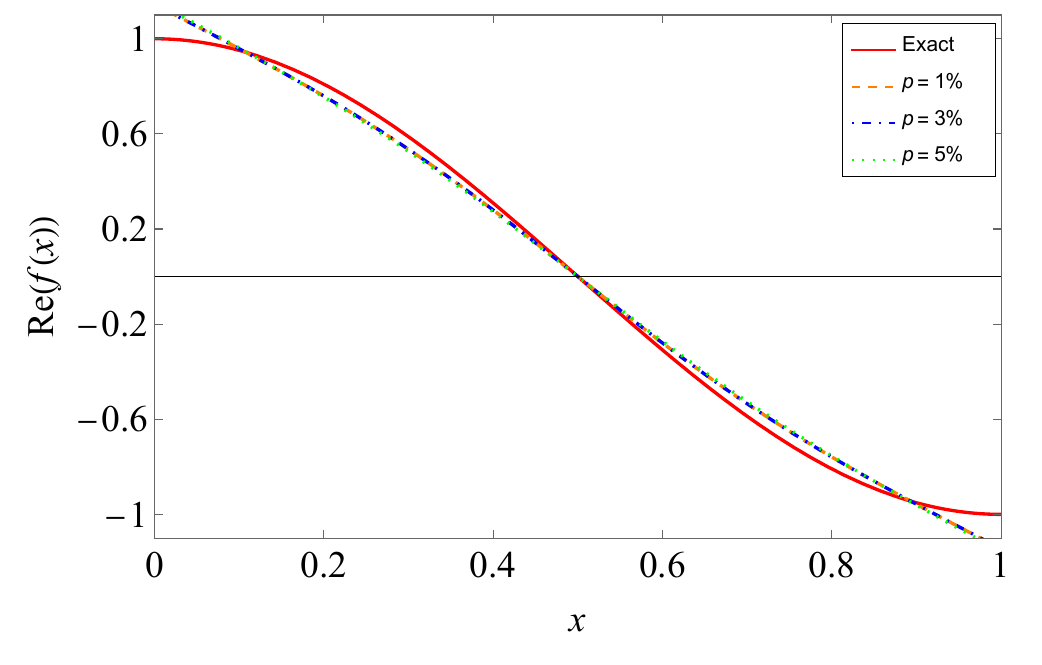} }}%
    \hspace{0.5cm}
    {{\includegraphics[width=7.3cm]{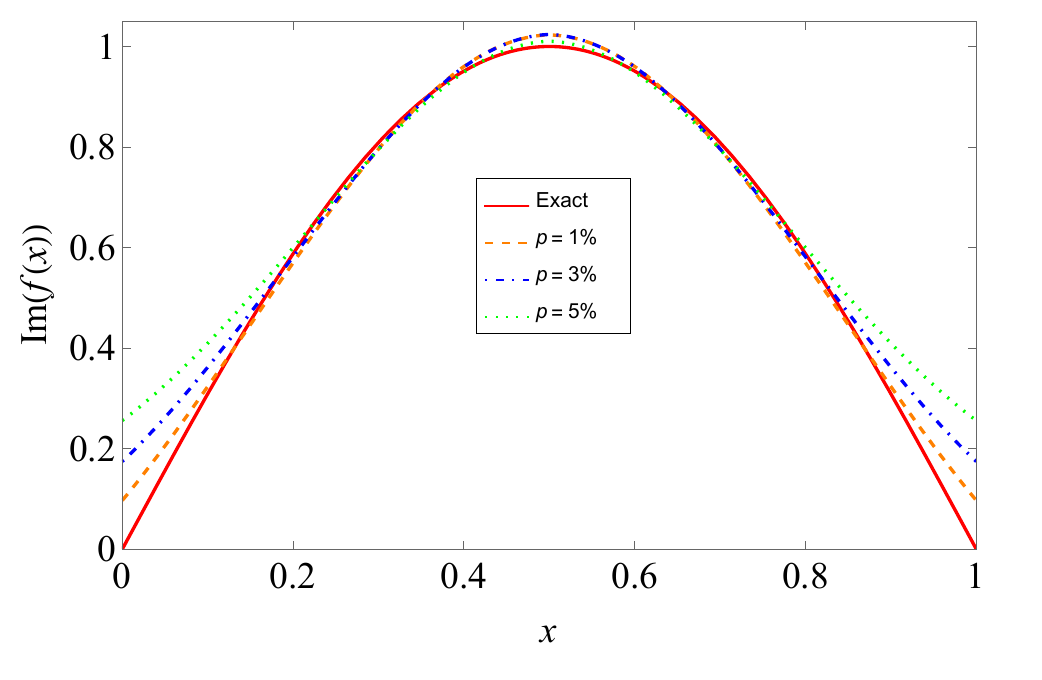} }}%
\caption{Exact and recovered $\Re f(x)$ (left) and $\Im f(x)$ (right) for $p\in \{1\%,3\%,5\%\}$.}\label{fig6}
\end{figure}

The numerical experiments presented in Figures \ref{fig2}, \ref{fig4}, and \ref{fig6}, demonstrate that the Landweber algorithm provides stable and accurate reconstruction for complex-valued functions. The algorithm captures both real and imaginary parts of the unknown source term. In addition, the recovery of the source function improves as the noise level $p$ decreases.

\section{Further comments and perspectives}
\label{section_further:comments}
We have investigated an inverse problem involving the determination of internal and boundary source terms in the Schr\"odinger equation with dynamic boundary conditions, using measured final time output data. We have analyzed a minimization problem for the associated Tikhonov functional and derived a gradient formula based on the solution of an appropriate adjoint system. Furthermore, we have established the Lipschitz continuity of the Fréchet gradient, which is crucial in gradient iterative methods.

The Landweber scheme we have considered provides stable and accurate results for reconstructing unknown source terms in the Schr\"odinger equation with dynamic boundary conditions. One could use other sophisticated gradient schemes such as conjugate gradient algorithms. Nevertheless, we chose the Landweber iteration as a basic scheme. Our approach can be extended to recover source terms in a multidimensional setting ($N=2$ or $3$). However, dynamic boundary conditions are more delicate to handle numerically in this case due to the geometry of the domain $\Omega$. This will be addressed in a forthcoming paper.

\appendix
\section{Appendix. Proof of Proposition \ref{proposition:wellposedness}}
\label{appendix:A}
In this appendix, we prove the well-posedness and regularity of solutions to problem \eqref{intro:eq:01}. To do this, we explore the existence of solutions via the semigroup approach together with some multiplier techniques applied for the Schr\"odinger equation with dynamic boundary conditions.
\subsection{Wentzell Laplacian}
We introduce the Wentzell-Laplacian operator $A_W:D(A_W)\subset \mathbb{L}^2 \to \mathbb{L}^2$ given by
\begin{align*}
    &A_W\left(
    \begin{array}{c}
    u\\u_\Gamma
    \end{array}
    \right)= \left(
    \begin{array}{c}
         -d\Delta u\\ d\pnu u-\gamma\Delta_\Gamma u_\Gamma
    \end{array}
    \right),\\
    &D(A_W)=\{(u,u_\Gamma)\in \mathbb{H}^1\,:\, (\Delta u, \gamma\Delta_\Gamma u_\Gamma-d\pnu u)\in \mathbb{L}^2\}=\mathbb{H}^2,
\end{align*}
where the last equality is obtained by the elliptic regularity.

Next, we summarize some properties of $A_W$ in the following proposition.
\begin{proposition}[See \cite{MMS'17}] The operator $-A_W$ is densely defined, self-adjoint, non-positive and generates an analytic $C_0$-semigroup $(e^{-tA_W})_{t\geq 0}$ on $\mathbb{L}^2$. Moreover, $D(A_W^{1/2})\equiv \mathbb{H}^1$.
\end{proposition}
Now, for $(\vec{r},\vec{r}_\Gamma)\in [\mathbb{W}^{1,\infty}]^N$ real-valued, let us consider the operator $A:D(A)\subset \mathbb{L}^2 \to \mathbb{L}^2$ defined by
\begin{align*}
    A\left(
    \begin{array}{c}
        u\\ u_\Gamma
    \end{array}
    \right)=\left(
    \begin{array}{c}
        di\Delta u -\vec{r}\cdot \nabla u\\
        -di\pnu u +\gamma i \Delta_\Gamma u_\Gamma -\vec{r}_\Gamma \cdot \nabla_\Gamma u_\Gamma
    \end{array}
    \right),\quad D(A)=D(A_W)=\mathbb{H}^2.
\end{align*}

\begin{proposition}
Suppose that $\vec{r} \cdot \nu \le 0$ on $\Gamma$. There exists a constant $\kappa \in \mathbb{R}$ which depends on $d,\gamma>0$, $\vec{r}$ and $\vec{r}_\Gamma$ such that the scaled operator $A-\kappa I_{\mathbb{H}^1}$ is maximal dissipative on $D(A_W^{1/2})$. Thus, $A$ generates a $C_0$-semigroup $(e^{tA})_{t\geq 0}$ on $\mathbb{H}^1$ satisfying
    \begin{align*}
        \|e^{tA}\|_{\mathbb{H}^1}\leq e^{\kappa t},\quad t\geq 0.
    \end{align*}
\end{proposition}
\begin{proof}
For each $U:=(u,u_\Gamma) \in D(A_W^{1/2})\equiv \mathbb{H}^1$, we have $$\langle AU,U \rangle_{D(A_W^{1/2})}=\langle AU,A_W U \rangle_{\mathbb{L}^2}+ \langle AU,U \rangle_{\mathbb{L}^2},$$ with
    \begin{align*}
        \langle AU,A_W U \rangle_{\mathbb{L}^2}
=& d\Re \int_\Omega (\vec{r}\cdot \nabla u)\Delta \overline{u}\,dx -d\Re \int_\Gamma (\vec{r}_\Gamma \cdot \nabla_\Gamma u_\Gamma)\pnu \overline{u}\,d\sigma \\
        &+\gamma\Re \int_{\Gamma} (\vec{r}_\Gamma \cdot \nabla_\Gamma u_\Gamma)\Delta_\Gamma \overline{u}_\Gamma d\sigma\\
        =&-d\Re \int_\Omega \left( \vec{\nabla}\vec{r}(\nabla u,\nabla \overline{u}) +\dfrac{1}{2}\vec{r} \cdot \nabla (|\nabla u|^2)\right)\,dx \\
        &\hspace{-0.7cm} -\gamma \Re \int_{\Gamma} \left(\vec{\nabla}_\Gamma \vec{r}_\Gamma (\nabla_\Gamma u_\Gamma ,\nabla_\Gamma u_\Gamma)+\frac{1}{2} \vec{r}_\Gamma \cdot \nabla_\Gamma (|\nabla_\Gamma u_\Gamma|^2) \right)\,d \sigma + d\int_\Gamma (\vec{r}\cdot \nu) |\partial_\nu u|^2\,d\sigma\\
        =& -d\Re \int_\Omega \vec{\nabla}\vec{r}(\nabla u,\nabla \overline{u})\,dx +\frac{d}{2} \int_\Omega \mathrm{div}(\vec{r}) |\nabla u|^2\, dx +\dfrac{d}{2} \int_\Gamma (\vec{r}\cdot \nu) |\partial_\nu u|^2\,d\sigma \\
        &\hspace{-1.3cm}- \frac{d}{2}\int_\Gamma (\vec{r}\cdot \nu) |\nabla_\Gamma u_\Gamma|^2\,d\sigma -\gamma\Re \int_\Gamma \vec{\nabla}_\Gamma \vec{r}_\Gamma (\nabla_\Gamma u_\Gamma ,\nabla_\Gamma \overline{u}_\Gamma)\,d \sigma +\frac{\gamma}{2}\int_\Gamma \mathrm{div}_\Gamma (\vec{r}_\Gamma) |\nabla_\Gamma u_\Gamma|^2\,d\sigma,
    \end{align*}
where we denote by $\vec{\nabla}\vec{r}$ and $\vec{\nabla}_\Gamma\vec{r}_\Gamma$ the Jacobian matrices in $\Omega$ and on $\Gamma$, respectively. We have used integration by parts in $\Omega$ and the surface divergence formula on $\Gamma$. The second term $\langle AU,U \rangle_{\mathbb{L}^2}$ can be computed similarly. Since $(\vec{r},\vec{r}_\Gamma)\in \mathbb{W}^{1,\infty}$, we can find $\kappa \in \mathbb{R}$ such that
    \begin{align*}
        \langle AU-\kappa U,U \rangle_{D(A_W^{1/2})}\leq 0\qquad \forall\,  U=\left(
        \begin{array}{c}
            u\\ u_\Gamma
        \end{array}
        \right)\in D(A_W^{1/2}).
    \end{align*}
On the other hand, to prove the maximal dissipativity of $A-\kappa I_{\mathbb{H}^1}$, we consider $\lambda>0$ and $G=(g,g_\Gamma)\in D(A_W^{1/2})$. Then, we need to prove the existence of a unique $F=(f,f_\Gamma)\in D(A)$ such that
    \begin{align}
        \label{eq:abstract:maximal:diss}
        [\lambda I -(A-\kappa I)]\left(
        \begin{array}{c}
            f\\ f_\Gamma
        \end{array}
        \right)=\left(
        \begin{array}{c}
            g\\g_\Gamma
        \end{array}
        \right).
    \end{align}
The equation \eqref{eq:abstract:maximal:diss} reads as follows:
    \begin{align}
        \begin{cases}
            -di\Delta f +\vec{r}\cdot \nabla f + (\lambda +\kappa)f=g &\text{ in }\Omega,\\
            di\pnu f -\gamma i \Delta_\Gamma f_\Gamma +\vec{r}_\Gamma \cdot \nabla f_\Gamma + (\lambda + \kappa)f_\Gamma =g_\Gamma &\text{ on }\Gamma,\\
            f\big|_\Gamma =f_\Gamma&\text{ on }\Gamma.
        \end{cases}
    \end{align}
By the elliptic regularity results for $\Delta$ and $\Delta_\Gamma$ (see \cite{MIR05}), we can ensure the existence of a unique solution $(f,f_\Gamma)$ with the desired properties. This provides the maximal dissipativity of $A-\kappa I_{\mathbb{H}^1}$.
\end{proof}


\subsection{Well-posedness}
Now, we focus on the non-homogeneous Cauchy problem
\begin{align}
    \label{non-homogeneous:prob:appendix}
    \begin{cases}
        \pt u-di\Delta u+\vec{\rho}\cdot \nabla u+\alpha u=g&\text{ in }Q,\\
        \pt u_\Gamma +di\pnu u -\gamma i \Delta_\Gamma u_\Gamma +\vec{\rho}_\Gamma \cdot \nabla_\Gamma u_\Gamma +\alpha_\Gamma u_\Gamma =g_\Gamma &\text{ on }\Sigma,\\
        u\big|_{\Gamma}=u_\Gamma &\text{ on }\Sigma ,\\
        (u(\cdot,0),u_\Gamma(\cdot,0))=(u_0,u_{\Gamma,0})&\text{ in }\Omega\times \Gamma.
    \end{cases}
\end{align}

\begin{proposition}
    Suppose that $(u_0,u_{\Gamma,0})\in \mathbb{L}^2$, $(\vec{\rho},\vec{\rho}_\Gamma)\in \mathbb{W}^{1,\infty}$ is real-valued, $(\alpha,\alpha_\Gamma)\in \mathbb{L}^\infty$ and $(g,g_\Gamma)\in L^1(0,T;\mathbb{L}^2)$. Then, there exists a constant $C>0$ such that the weak solution $(u,u_\Gamma)\in C^0([0,T];\mathbb{L}^2)$ of \eqref{non-homogeneous:prob:appendix} satisfies
    \begin{align}
        \label{prop:L2:energy:ineq}
        \|(u,u_\Gamma)\|_{C^0([0,T];\mathbb{L}^2)}\leq C\left(\|(u_0,u_{\Gamma,0})\|_{\mathbb{L}^2} +  \|(g,g_\Gamma)\|_{L^1(0,T;\mathbb{L}^2)} \right).
    \end{align}
\end{proposition}

\begin{proof}
    We argue by density. Firstly, we multiply the first equation of \eqref{non-homogeneous:prob:appendix} by $\overline{u}$ and integrate on $\Omega$. Secondly, we multiply the second equation of \eqref{non-homogeneous:prob:appendix} by $\overline{u}_\Gamma$ and integrate on $\Gamma$. Next, we add up these identities and take the real part of the obtained equation. Then, the following identity holds for a.e. $t\in (0,T)$:
    \begin{align}
        \label{prop:existence:L2:01}
        \begin{split}
        &\Re \int_\Omega \overline{u}\pt u\,dx + d\Im \int_\Omega \overline{u} \Delta u\, dx + \Re \int_\Omega \overline{u}(\vec{\rho}\cdot \nabla u)\,dx +\Re \int_\Omega \alpha |u|^2\,dx\\
        &+\Re \int_\Gamma \overline{u}_\Gamma \pt u_\Gamma \,d\sigma -d\Im \int_{\Gamma} \overline{u}_\Gamma \pnu u\,d\sigma +\gamma\Im \int_{\Gamma} \overline{u}_\Gamma \Delta_\Gamma u_\Gamma \,d\sigma \\
        &+\Re \int_{\Gamma} \overline{u}_\Gamma (\vec{\rho}_\Gamma \cdot \nabla_\Gamma u_\Gamma)\,d\sigma +\Re \int_{\Gamma} \alpha_{\Gamma}|u_\Gamma|^2\,d\sigma\\
        =&\Re \int_\Omega \overline{u}g\,dx +\Re \int_\Gamma \overline{u}_\Gamma g_\Gamma \,d\sigma.
        \end{split}
    \end{align}
    Notice that
    \begin{align}
        \label{prop:existence:L2:02}
        \Re \int_\Omega \overline{u}\pt u\,dx +\Re \int_\Gamma \overline{u}_\Gamma \pt u_\Gamma\,d\sigma =\dfrac{d}{dt}\left( \dfrac{1}{2}\int_\Omega |u|^2\,dx + \dfrac{1}{2} \int_{\Gamma} |u_\Gamma|^2\,d\sigma \right).
    \end{align}
    Moreover, integration by parts in $\Omega$ and the surface divergence theorem on $\Gamma$ imply that
    \begin{align}
        \label{prop:existence:L2:03}
        d\Im \int_\Omega \overline{u}\Delta u\,dx -d\Im \int_{\Gamma} \overline{u}_\Gamma \pnu u\,d\sigma +\gamma \Im \int_{\Gamma} \overline{u}_\Gamma \Delta_\Gamma u_\Gamma \,d\sigma =0,
    \end{align}
    where we have used the side condition $u =u_\Gamma$ on $\Sigma$. On the other hand,
    \begin{align}
        \label{prop:existence:L2:04}
        \begin{split}
        &\Re \int_\Omega \overline{u}(\vec{\rho}\cdot \nabla u)\,dx +\Re \int_\Gamma \overline{u}_\Gamma (\vec{\rho}_\Gamma \cdot \nabla_\Gamma u_\Gamma)\,d\sigma \\
        =&-\frac{1}{2}\int_\Omega \mathrm{div}(\vec{\rho})|u|^2\,dx -\frac{1}{2}\int_{\Gamma} \mathrm{div}_\Gamma(\vec{\rho}_\Gamma) |u_\Gamma|^2\,d\sigma +\frac{1}{2}\int_{\Gamma}(\vec{\rho}\cdot \nu) |u|^2d\sigma .
        \end{split}
    \end{align}
    Substituting \eqref{prop:existence:L2:02}-\eqref{prop:existence:L2:04} into \eqref{prop:existence:L2:01} and applying the Cauchy-Schwarz inequality we deduce
    \begin{align*}
        \frac{d}{dt} \|(u(\cdot,t),u_\Gamma(\cdot,t))\|_{\mathbb{L}^2}^2 \leq C\| (u(\cdot,t),u_\Gamma(\cdot,t)) \|_{\mathbb{L}^2}^2 + C\|(u(\cdot,t),u_\Gamma(\cdot,t))\|_{\mathbb{L}^2} \|(g(\cdot,t),g_\Gamma (\cdot,t))\|_{\mathbb{L}^2}.
    \end{align*}
Then, applying Gronwall's and H\"older's inequalities, we finally deduce the inequality \eqref{prop:L2:energy:ineq}.
\end{proof}

\begin{proposition}
    \label{appendix:wellposedness}
    Suppose that $(\vec{\rho},\vec{\rho}_\Gamma)\in \mathbb{W}^{1,\infty}$ is real-valued and $(\alpha,\alpha_\Gamma)\in \mathbb{L}^\infty$ such that $\vec{\rho}\cdot \nu \leq 0$ on $\Gamma$. In addition, suppose that $(u_0,u_{\Gamma,0})\in \mathbb{H}^1$ and $(g,g_\Gamma)\in L^1(0,T;\mathbb{H}^1)$. Then, there exists a unique solution $(u,u_\Gamma) \in C^0([0,T];\mathbb{H}^1)$ of \eqref{non-homogeneous:prob:appendix}. Moreover, there exists a constant $C>0$ such that
    \begin{align}
        \label{prop:existence:H1:ineq}
        \|(u,u_\Gamma)\|_{C^0([0,T];\mathbb{H}^1)} \leq C\left(\|(u_0,u_{\Gamma,0})\|_{\mathbb{H}^1} + \|(g,g_\Gamma)\|_{L^1(0,T;\mathbb{H}^1)} \right).
    \end{align}
\end{proposition}

\begin{proof}
    We argue by density. First, we multiply the first equation of \eqref{non-homogeneous:prob:appendix} by $-d\Delta \overline{u}$ and integrate in $\Omega$. Secondly, we multiply the second equation of \eqref{non-homogeneous:prob:appendix} by $(-\gamma \Delta_\Gamma \overline{u}_\Gamma +d\pnu \overline{u})$ and integrate on $\Gamma$. Next, adding up these identities and taking the real part we have
    \begin{align}
        \label{prop:existence:H1:01}
        \begin{split}
        &-d\Re \int_\Omega \pt u\Delta \overline{u}\,dx -d\Re \int_\Omega (\vec{\rho}\cdot \nabla u)\Delta \overline{u}\,dx -d\Re \int_\Omega \alpha u \Delta \overline{u}\,dx \\
        &+\Re\int_\Gamma \pt u_\Gamma (-\gamma \Delta_\Gamma \overline{u}_\Gamma +d\pnu \overline{u})\,d\sigma +\Re \int_\Gamma (\vec{\rho}\cdot \nabla_\Gamma u_\Gamma)(-\gamma \Delta_\Gamma \overline{u}_\Gamma +d\pnu \overline{u})\,d\sigma \\
        &+\Re \int_{\Gamma}\alpha_\Gamma u_\Gamma (-\gamma \Delta_\Gamma \overline{u}_\Gamma +d\pnu \overline{u})\,d\sigma \\
        =&-d\Re \int_\Omega g\Delta \overline{u}\,dx + \Re \int_\Gamma g_\Gamma (-\gamma \Delta_\Gamma \overline{u}_\Gamma +d\pnu \overline{u})\,d\sigma .
        \end{split}
    \end{align}
    Now, integration by parts in $\Omega$ and the surface divergence theorem on $\Gamma$ imply that
    \begin{align}
        \label{prop:existence:H1:02}
        \begin{split}
            &-d\Re \int_\Omega \pt u\Delta \overline{u}\,dx +\Re \int_{\Gamma} \pt u_\Gamma (-\gamma \Delta_\Gamma \overline{u}_\Gamma +d\pnu \overline{u})\,d\sigma \\
            =&\dfrac{d}{dt} \left(\frac{d}{2}\int_\Omega |\nabla u|^2\,dx + \frac{\gamma}{2} \int_\Gamma |\nabla_\Gamma u_\Gamma|^2\,d\sigma \right).
        \end{split}
    \end{align}
    In addition,
    \begin{align*}
        \mu:=&-d\Re \int_\Omega (\vec{\rho}\cdot \nabla u)\Delta \overline{u}\,dx + \Re\int_{\Gamma} (\vec{\rho}_\Gamma \cdot \nabla_\Gamma u_\Gamma)(-\gamma \Delta_\Gamma \overline{u}_\Gamma +d\pnu \overline{u})\,d\sigma\\
        =&d\Re \int_\Omega \nabla (\vec{\rho}\cdot \nabla u)\cdot \nabla \overline{u}\,dx -d\Re\int_{\Gamma} (\vec{\rho}\cdot \nabla u)\pnu \overline{u}\, d\sigma \\
        &+\gamma \Re\int_{\Gamma}\nabla_\Gamma (\vec{\rho}_\Gamma \cdot \nabla_\Gamma u_\Gamma)\cdot \nabla_\Gamma \overline{u}_\Gamma \,d\sigma +d\Re \int_{\Gamma} (\vec{\rho}_\Gamma \cdot \nabla_\Gamma u_\Gamma)\pnu \overline{u}\,d\sigma .
    \end{align*}
    Since
    \begin{align*}
        &\nabla(\vec{\rho}\cdot \nabla u)\cdot \nabla\overline{u}=\vec{\nabla}\vec{\rho}(\nabla u,\nabla \overline{u}) +\frac{1}{2}\vec{\rho}\cdot \nabla (|\nabla u|^2),\\
        \nabla_\Gamma (\vec{\rho}_\Gamma &\cdot\nabla_\Gamma u_\Gamma)\cdot \nabla_\Gamma \overline{u}_\Gamma =\vec{\nabla}_\Gamma \vec{\rho}_\Gamma (\nabla_\Gamma u_\Gamma,\nabla_\Gamma \overline{u}_\Gamma) + \frac{1}{2}\vec{\rho}_\Gamma \cdot \nabla_\Gamma (|\nabla_\Gamma u_\Gamma|^2),
    \end{align*}
    and that $\nabla u=\nabla_\Gamma u+(\pnu u)\nu$ on $\Gamma$, we obtain
    \begin{align}
        \label{prop:existence:H1:03}
        \begin{split}
            \mu:= &d\Re \int_\Omega \vec{\nabla}\vec{\rho}(\nabla u,\nabla\overline{u})\,dx +\gamma \Re \int_{\Gamma} \vec{\nabla}\vec{\rho}_\Gamma (\nabla_\Gamma u_\Gamma,\nabla_\Gamma \overline{u}_\Gamma)\,d\sigma \\
            &-\frac{d}{2} \int_\Omega \mathrm{div}(\vec{\rho})|\nabla u|^2\,dx -\frac{\gamma}{2} \int_\Gamma \mathrm{div}_\Gamma (\vec{\rho}_\Gamma) |\nabla_\Gamma u_\Gamma|^2\,d\sigma \\
            &-\frac{d}{2}\int_\Gamma (\vec{\rho}\cdot \nu) |\pnu u|^2\,d\sigma +\frac{d}{2}\int_{\Gamma}(\vec{\rho}\cdot \nu) |\nabla_\Gamma u_\Gamma|^2\,d\sigma,
        \end{split}
    \end{align}
    where we have used integration by parts and the fact that $\vec{\rho}=\vec{\rho}_\Gamma$ on $\Gamma$. In addition, we have
    \begin{align}
        \label{prop:existence:H1:04}
        \begin{split}
            &-d\Re \int_\Omega \alpha u\Delta \overline{u}\,dx +\Re \int_\Gamma \alpha_\Gamma u_\Gamma (-\gamma \Delta_\Gamma \overline{u}_\Gamma +d\pnu \overline{u})\,d\sigma \\
            =&d \Re\int_\Omega \alpha |\nabla u|^2\,dx +\gamma \Re\int_{\Gamma} \alpha_\Gamma |\nabla_\Gamma u_\Gamma|^2\,d\sigma\\
            &+d\Re \int_\Omega u\nabla\alpha\cdot \nabla \overline{u}\,dx + \gamma \Re \int_{\Gamma}u_\Gamma \nabla_\Gamma \alpha_\Gamma \cdot \nabla_\Gamma \overline{u}_\Gamma \, d\sigma .
        \end{split}
    \end{align}
Since $g=g_\Gamma$ on $\Sigma$, we deduce that
    \begin{align}
        \label{prop:existence:H1:05}
        \begin{split}
            &-d\Re\int_\Omega g\Delta \overline{u}\,dx + \Re \int_{\Gamma}g_\Gamma (-\gamma \Delta_\Gamma \overline{u}_\Gamma +d\pnu \overline{u})\,d\sigma \\
            =&d\Re \int_\Omega \nabla g\cdot \nabla \overline{u}\,dx + \gamma \Re\int_{\Gamma} \nabla_\Gamma g_\Gamma \cdot \nabla_\Gamma \overline{u}_\Gamma\,d\sigma .
        \end{split}
    \end{align}
    Thus, substituting \eqref{prop:existence:H1:02}-\eqref{prop:existence:H1:05} into \eqref{prop:existence:H1:01} and using the fact that $\vec{\rho}\cdot \nu \leq 0$ on $\Gamma$, we can assert that
    \begin{align*}
        \frac{d}{dt}\|(u(\cdot,t),u_\Gamma(\cdot,t))\|_{\mathbb{H}^1}^2 \leq C\|(u(\cdot,t),u_\Gamma(\cdot,t))\|_{\mathbb{H}^1}^2 + C\|(u(\cdot,t),u_\Gamma(\cdot,t))\|_{\mathbb{H}^1}\cdot \|(g,g_\Gamma)\|_{\mathbb{H}^1}.
    \end{align*}
    Finally, by Gronwall's and H\"older's inequalities we obtain \eqref{prop:existence:H1:ineq}.
\end{proof}

\subsection{Hidden regularity}

\begin{proposition}
	\label{Proposition:previous:hidden}
    Consider $\vec{q}\in C^2(\overline{Q};\mathbb{R}^N)$ and consider the same assumptions of Proposition \ref{appendix:wellposedness}. Then, we have the following identity
    \begin{align}
        \dfrac{d}{2}\iint_\Sigma (\vec{q}\cdot \nu) |\pnu u|^2\,d\sigma \,dt=\mathcal{X}+\mathcal{Y}+\mathcal{G},
    \end{align}
    where $\mathcal{X},\mathcal{Y}$ and $\mathcal{G}$ are given by
    \begin{align*}
        \mathcal{X}:=&-\frac{\gamma}{2} \iint_{\Sigma} (\vec{q}\cdot \nu) |\nabla_\Gamma u_\Gamma|^2\,d\sigma\, dt -\frac{1}{2}\Im \iint_{\Sigma} (\vec{q}\cdot \nu)\alpha_{\Gamma} |u_\Gamma|^2\,d\sigma \,dt \\
        &+\frac{d}{2}\iint_{Q} \vec{\nabla}\vec{q}(\nabla u,\nabla \overline{u})\,dx \,dt + \frac{d}{2}\iint_{Q}\mathrm{div}(\vec{q})|\nabla u|^2\,dx\,dt\\
        &+\frac{\gamma}{2} \iint_{\Sigma} \vec{\nabla}_\Gamma \vec{q}(\nabla_\Gamma u_\Gamma ,\nabla_\Gamma u_\Gamma)\, d\sigma\, dt +\frac{\gamma}{2} \iint_{\Sigma} \mathrm{div}_\Gamma(\vec{q})|\nabla_\Gamma u_\Gamma|^2\,d \sigma \, dt\\
        &+\Im \iint_{Q} (\vec{\rho}\cdot \nabla u)(\vec{q}\cdot \nabla \overline{u})\,dx\,dt +\frac{1}{2}\Im \iint_{Q} \alpha\mathrm{div}(\vec{q}) |u|^2\,dx\,dt \\
        &+\Im \iint_{\Sigma} (\vec{\rho}_\Gamma \cdot \nabla_\Gamma u_\Gamma)(\vec{q}\cdot \nabla_\Gamma \overline{u}_\Gamma)\,d\sigma \,dt +\frac{1}{2}\Im \iint_\Sigma \alpha_{\Gamma} \mathrm{div}_\Gamma(\vec{q})|u_\Gamma|^2\,d\sigma \,dt .
    \end{align*}
    \begin{align*}
        \mathcal{Y}:=&\frac{1}{2}\Im \iint_{Q} \overline{u}\nabla u\cdot \pt \vec{q}\,dx\,dt -\frac{1}{2}\Im \int_\Omega \overline{u}\nabla u\cdot \vec{q}\,\bigg|_0^T \,dx \\
        &-\frac{d}{2}\Re \iint_{\Sigma} (\vec{q}\cdot \nu)\overline{u}_\Gamma \pnu u \,d\sigma \, dt -\frac{1}{2}\Im \iint_{\Sigma} (\vec{q}\cdot \nabla_\Gamma \nu) (\vec{\rho}_\Gamma \cdot \nabla u_\Gamma) \overline{u}_\Gamma \,d\sigma\, dt\\
        &+\frac{1}{2}\Im \iint_{\Sigma} (\nabla_\Gamma u_\Gamma \cdot \pt \vec{q})\overline{u}_\Gamma \,d\sigma\, dt -\frac{1}{2}\Im \iint_{\Sigma}(\vec{q}\cdot \nabla_\Gamma u_\Gamma)\overline{u}_\Gamma \,\bigg|_0^T d\sigma \\
        &+\frac{d}{2}\Re \iint_{Q} \overline{u}\nabla (\mathrm{div}(\vec{q}))\cdot \nabla u\,dx\,dt +\dfrac{1}{2}\Re \iint_{\Sigma} \pnu u \vec{q}\cdot \nabla_\Gamma u_\Gamma \, d\sigma \, dt\\
        &+\frac{\gamma}{2} \Re \iint_\Sigma \overline{u}_\Gamma \nabla_\Gamma  (\mathrm{div}_\Gamma(\vec{q}))\cdot \nabla_\Gamma u_\Gamma\,d\sigma\,dt +\frac{1}{2}\Im \iint_{Q} \mathrm{div}(\vec{q})(\vec{\rho}\cdot \nabla u)\overline{u}\,dx\,dt \\
        &+\Im \iint_{Q} \alpha(\vec{q}\cdot \nabla \overline{u})u\,dx\,dt +\frac{1}{2}\Im\iint_{\Sigma} \mathrm{div}_\Gamma(\vec{q})(\vec{\rho}_\Gamma \cdot \nabla_\Gamma u_\Gamma)\overline{u}_\Gamma \,d\sigma \,dt\\
        &+\Im \iint_{\Sigma} \alpha_\Gamma (\vec{q}\cdot \nabla_\Gamma \overline{u}_\Gamma)u_\Gamma d\sigma dt
    \end{align*}
    and
    \begin{align*}
        \mathcal{G}:=&\frac{1}{2}\Im \iint_{\Sigma} (\vec{q}\cdot \nu)\overline{u}_\Gamma g_\Gamma d\sigma dt -\Im \iint_{Q} g\left(\vec{q}\cdot\nabla \overline{u}+\frac{1}{2}\mathrm{div}(\vec{q}) \overline{u}\right)\,dx\,dt \\
        &-\Im \iint_{\Sigma} g_\Gamma \left(\vec{q}\cdot \nabla_\Gamma u_\Gamma +\frac{1}{2}\mathrm{div}_\Gamma(\vec{q})\overline{u}_\Gamma \right)\,d\sigma\,dt .
    \end{align*}
\end{proposition}
\begin{proof}
    The proof follows from multiplier techniques for the Schr\"odinger equation; see for instance \cite{Mach94} for Dirichlet boundary conditions (see also \cite{Zua03} for a concise review). However, as we deal with dynamic boundary conditions, we should handle several new boundary terms. For the sake of completeness, we sketch the main steps of the proof.

    We argue by density. We multiply the first equation of \eqref{non-homogeneous:prob:appendix} by $-(\vec{q}\cdot \nabla \overline{u} +\frac{1}{2}\mathrm{div}(\vec{q})\overline{u})$ and integrate in $Q$. Similarly, we multiply the second equation by $-(\vec{q}\cdot \nabla_\Gamma \overline{u}_\Gamma +\frac{1}{2}\mathrm{div}_\Gamma(\vec{q})\overline{u}_\Gamma)$ and integrate the results on $\Sigma$. Next, we add up both identities and take the imaginary part. After that, we use integration by parts and the surface divergence theorem to deal with the terms involving $\Delta u$, $\Delta_\Gamma u_\Gamma$, $\pt u$ and $\pt u_\Gamma$. The rest of the proof runs as in \cite{Mach94}.
\end{proof}

\begin{lemma}
\label{Proposition:Hidden:regularity}
Assume the same conditions of Proposition \ref{appendix:wellposedness}. Then, the normal derivative $\pnu u$ associated with the solution of \eqref{non-homogeneous:prob:appendix} belongs to $L^2(\Sigma)$. Moreover, there exists $C>0$ which depends on $d,\gamma,\rho,\rho_\Gamma, \alpha,\alpha_\Gamma$ such that
    \begin{align*}
    	\|\pnu u\|_{L^2(\Sigma)} \leq C\left(\|(u_0,u_{\Gamma,0})\|_{\mathbb{H}^{1}} +\|(g,g_\Gamma)\|_{L^1(0,T;\mathbb{H}^1)}\right).
    \end{align*}
\end{lemma}

\begin{proof}
	Let us choose $\vec{q}\in C^2(\overline{\Omega};\mathbb{R}^N)$ such that $q=\nu$ on $\Gamma$. The existence of this vector field is guaranteed (see \cite{JLLions}). Then, by Proposition \ref{Proposition:previous:hidden}, we see that
	\begin{align}
		\label{eq:hidden:ineq}
		\frac{d}{2}\|\pnu u\|_{L^2(\Sigma)}^2 =\mathcal{X}+ \mathcal{Y}+\mathcal{G}.
	\end{align}
Now, since $(\rho,\rho_\Gamma)$, $\vec{q}$ and their derivatives are bounded, we easily deduce that
	\begin{align}
		\nonumber
		|\mathcal{X}|\leq &C\|(u,u_\Gamma)\|^2_{C^0([0,T];\mathbb{H}^1)}\\
		\label{estimate:mathcalX}
		\leq & C \left( \|(u_0,u_{\Gamma,0})\|_{\mathbb{H}^1}^2 + \|(g,g_\Gamma)\|_{L^1(0,T;\mathbb{H}^1)}^2 \right),
	\end{align}
where we have used Proposition \ref{appendix:wellposedness}. On the other hand, by Young's inequality, it is easy to see that $\mathcal{Y}$ can be estimated in the following way:
	\begin{align}
		\label{estimate:mathcalY}
		|\mathcal{Y}|\leq C\|(u_0,u_{\Gamma,0})\|_{\mathbb{H}^1}^2 + C\|(g,g_\Gamma)\|_{L^1(0,T;\mathbb{H}^1)}^2.
	\end{align}
In a similar way, by Cauchy-Scharwz inequality we can estimate $\mathcal{G}$ as follows:
	\begin{align}
		\label{estimate:mathcalG}
		|\mathcal{G}|\leq C\|(u_0,u_{\Gamma,0})\|_{\mathbb{H}^1}^2 + C\|(g,g_\Gamma)\|_{L^1(0,T;\mathbb{H}^1)}^2.
	\end{align}
Finally, substituting \eqref{estimate:mathcalX}-\eqref{estimate:mathcalG} into \eqref{eq:hidden:ineq}, we get our desired conclusion.
\end{proof}

\section*{Acknowledgment}
The basic idea of this work started when the first and third named authors attended the conference ``X Partial differential equations, optimal design and numerics (2024)” held at the `Centro de Ciencias de Benasque Pedro Pascual’. The authors thank the organizers and the center for their kind hospitality. The research of the second author has been supported by the Scientific and Technological Research Council of Türkiye (TUBITAK) through the Incentive Program for International Scientific Publications (UBYT). The third author has been funded under the Grant QUALIFICA by Junta de Andaluc\'ia grant number QUAL21 005 USE.



\begin{thebibliography}{99}
\bibitem{ACM'21}
E. M. Ait Ben Hassi, S. E. Chorfi and L. Maniar, An inverse problem of radiative potentials and initial temperatures in parabolic equations with dynamic boundary conditions, \emph{J. Inverse Ill-Posed Probl.}, \textbf{30} (2022), 363--378.

\bibitem{ACM'22}
E. M. Ait Ben Hassi, S. E. Chorfi and L. Maniar,
Identification of source terms in heat equation with dynamic boundary conditions,
\emph{Math. Meth. Appl. Sci.}, \textbf{45} (2022), 2364--2379.

\bibitem{ACMO'21}
E. M. Ait Ben Hassi, S. E. Chorfi, L. Maniar and O. Oukdach,
Lipschitz stability for an inverse source problem in anisotropic parabolic equations with dynamic boundary conditions,
\emph{Evol. Equ. and Cont. Theo.}, \textbf{10} (2021), 837--859.

\bibitem{AM12}
S. A. Avdonin and V. S. Mikhaylov, Inverse source problem for the 1D Schrödinger equation, \emph{J. Math. Sci.}, \textbf{185} (2012), 513--516.

\bibitem{AMR14}
S. A. Avdonin, V. S. Mikhaylov and K. Ramdani, Reconstructing the potential for the one-dimensional Schrödinger equation from boundary measurements, \emph{IMA J. Math. Contr. Inf.}, \textbf{31} (2014), 137--150,

\bibitem{Ba93}
A. D. Bandrauk, \emph{Molecules in laser fields}, CRC Press, 1993.

\bibitem{BM08}
L. Baudouin and A. Mercado, An inverse problem for Schr\"odinger equations with discontinuous
main coefficient, \emph{Appl. Anal.}, \textbf{87} (2008), 1145--1165.

\bibitem{BP08}
L. Baudouin and J.-P. Puel, Corrigendum Uniqueness and stability in an inverse problem for the Schr\"odinger equation, \emph{Inverse Problems}, \textbf{23} (2008), 1327--1328.

\bibitem{BP02}
L. Baudouin and J.-P. Puel, Uniqueness and stability in an inverse problem for the Schr\"odinger equation, \emph{Inverse Problems}, \textbf{18} (2002), 1537–1554.

\bibitem{BS12}
F. A. Berezin and M. Shubin, \emph{The Schrödinger Equation}, Vol. \textbf{66}, Springer Science \& Business Media, 2012.

\bibitem{CMM23}
N. Carreño, A. Mercado and R. Morales, Local null controllability of a cubic Ginzburg-Landau equation with dynamic boundary conditions, \emph{arXiv:2301.03429} (2023).

\bibitem{CCLL16}
M. M. Cavalcanti, W. J. Corr\^ea, I. Lasiecka and C. Lefler, Well-posedness and uniform stability for nonlinear Schr\"odinger equations with dynamic/Wentzell boundary conditions, \emph{Indiana Univ. Math. J.}, \textbf{65} (2016), 1445--1502.

\bibitem{CL0623}
Q. Chen and J. Liu, Solving an inverse parabolic problem by optimization from final measurement data, \emph{J. Comput. Appl. Math.}, \textbf{193} (2006), 183--203.

\bibitem{CGMZ23}
S. E. Chorfi, G. El Guermai, L. Maniar and W. Zouhair, Numerical identification of initial temperatures in heat equation with dynamic boundary conditions, \emph{Mediterranean J. Math.}, \textbf{20} (2023), 256.

\bibitem{CGMZ232}
S. E. Chorfi, G. El Guermai, L. Maniar and W. Zouhair, Identification of source terms in wave equation with dynamic boundary conditions, \emph{Math. Meth. Appl. Sci.}, \textbf{46} (2023), 911--929.

\bibitem{Eg18}
H. Egger, K. Fellner, J. F. Pietschmann and B. Q. Tang, Analysis and numerical solution of coupled volume-surface reaction–diffusion systems with application to cell biology, \emph{Appl. Math. Comput.}, \textbf{336} (2018), 351--367.

\bibitem{EHN'00}
H. W. Engl, M. Hanke and A. Neubauer,
\emph{Regularization of Inverse Problems},
Kluwer Academic Publishers, 2000.

\bibitem{GM92}
A. Giusti-Suzor and F. H. Mies, Vibrational trapping and suppression of dissociation in intense laser fields, \emph{Physical review letters}, \textbf{68} (1992), 3869.

\bibitem{GM13}
A. Glitzky and A. Mielke, A gradient structure for systems coupling reaction-diffusion effects in bulk and interfaces, \emph{Z. angew. Math. Physik.}, \textbf{64} (2013), 29--52.

\bibitem{Go'06}
G. R. Goldstein, Derivation and physical interpretation of general boundary conditions, \emph{Adv. Diff. Equ.}, \textbf{11} (2006), 457--480.

\bibitem{Ha09}
A. Hasanov, Identification of an unknown source term in a vibrating cantilevered beam from final overdetermination, \emph{Inverse Problems}, \textbf{25} (2009), 115015.

\bibitem{Ha'07}
A. Hasanov,
Simultaneous determination of source terms in a linear parabolic problem from the final overdetermination: Weak solution approach,
\emph{J. Math. Anal. Appl.}, \textbf{330} (2007), 766--779.

\bibitem{Ha091}
A. Hasanov, Simultaneous determination of source terms in a linear hyperbolic problem from the final overdetermination: weak solution approach, \emph{IMA J. Appl. Math.}, \textbf{74} (2009), 1--19.

\bibitem{Ha19}
A. Hasanov, O. Baysal and C. Sebu, Identification of an unknown shear force in the Euler-Bernoulli cantilever beam from measured boundary deflection, \emph{Inverse Problems}, \textbf{35} (2019), 115008.

\bibitem{HR'21}
A. Hasanov Hasanoğlu and V. G. Romanov,
\emph{Introduction to Inverse Problems for Differential Equations},
Springer, second edition, 2021.

\bibitem{IY24}
O. Y. Imanuvilov and M. Yamamoto, Determination of a source term in Schr\"odinger equation with data taken at final moment of observation, \emph{Commun. Anal. Comp.}, (2024), doi: 10.3934/cac.2024016.

\bibitem{La32}
R. E. Langer, A problem in diffusion or in the flow of heat for a solid in contact with a fluid, \emph{Tohoku Math. J.}, \textbf{35} (1932), 260--275.

\bibitem{Iv63}
V. K. Ivanov, On ill-posed problems, \emph{Mat. USSR Sb.}, \textbf{61} (1963), 211--223.

\bibitem{LTZ04}
I. Lasiecka, R. Triggiani and X. Zhang, Global uniqueness, observability and stabilization
of nonconservative Schr\"odinger equations via pointwise Carleman estimates. I. $H^1(\Omega)$-estimates, \emph{J. Inverse Ill-Posed Probl.}, \textbf{12} (2004), 43--123.

\bibitem{JLLions}
J.-L. Lions, \emph{Contr\^olabilit\'e Exacte, Perturbations et Stabilisation de Syst\`emes Distribu\'es}, Tome
1, Rech. Math. Appl., \textbf{8}, Masson, Paris, 1988.

\bibitem{Mach94}
E. Machtyngier, Exact controllability for the Schrödinger equation, \emph{SIAM J. Control Optim.}, \textbf{32} (1994), 24--34.

\bibitem{MMS'17}
L. Maniar, M. Meyries and R. Schnaubelt,
Null controllability for parabolic equations with dynamic boundary conditions,
\emph{Evol. Equat. and Cont. Theo.}, \textbf{6} (2017), 381--407.

\bibitem{MM23}
A. Mercado and R. Morales, Exact controllability for a Schrödinger equation with dynamic boundary conditions, \emph{SIAM J. Control Optim.}, \textbf{61} (2023), 3501--3525.

\bibitem{MOR08}
A. Mercado, A. Osses and L. Rosier, Inverse problems for the Schr\"odinger equation via Carleman inequalities with degenerate weights, \emph{Inverse Problems}, \textbf{24} (2008), 015017.

\bibitem{MIR05}
A. Miranville and S. Zelik, Exponential attractors for the Cahn–Hilliard equation with dynamic boundary conditions, \emph{Math. Meth. Appl. Sci.}, \textbf{28} (2005), 709--735.

\bibitem{Nem86}
A. C. Nemirovskii, The regularizing properties of the adjoint gradient method in ill-posed problems. \emph{USSR Comput. Math. Math. Phys.}, \textbf{26} (1986), 7--16.

\bibitem{Sa20}
N. Sauer, Dynamic boundary conditions and the Carslaw-Jaeger constitutive relation in heat transfer, \emph{SN Partial Differ. Equ. Appl.}, \textbf{1} (2020), 48.

\bibitem{YFL19}
F. Yang, J. L. Fu, X. X. Li, A potential-free field inverse Schrödinger problem: optimal error bound analysis and regularization method, \emph{Inv. Probl.  Sci. Eng.}, \textbf{28} (2019), 1209--1252.

\bibitem{Zua03}
E. Zuazua, Remarks on the controllability of the Schr\"odinger equation, Quantum control: mathematical and numerical challenges, 193--211, CRM Proc. Lecture Notes, \textbf{33}, Amer. Math. Soc., Providence, RI (2003).

\end{thebibliography}
\end{document}